\documentclass[11pt,a4paper]{amsart}
\usepackage{amssymb,amsmath}
\def\({\left(}
\def\){\right)}
\def\Nx{\nabla_x}
\def\Cal{\mathcal}
\def\Om{\Omega}

\def\eb{\varepsilon}
\def\dist{{\rm dist}}

\def\R {\mathbb{R}}

\newcommand{\be}{\begin{equation} }
\newcommand{\ee}{\end{equation} }

\def \p {\partial}

\def \and{\qquad\text{and}\qquad}

\def\Bbb{\mathbb}
\def\Dt{\partial_t}
\def\Dx{\Delta_x}
\def\({\left(}
\def\){\right)}
\def\Nx{\nabla_x}
\def\divv{\operatorname{div}}
\def\eb{\varepsilon}
\def\Cal{\mathcal}
\def\eb{\varepsilon}

\def\Om{\Omega}

\def\dist{{\rm dist}}

\def\R {\mathbb{R}}

\def\<{\left<}
\def\>{\right>}

\def \p {\partial}

\def \and{\qquad\text{and}\qquad}

\def\Bbb{\mathbb}
\def\Dt{\partial_t}
\def\Dx{\Delta_x}
\newtheorem{proposition}{Proposition}[section]
\newtheorem{theorem}[proposition]{Theorem}
\newtheorem{corollary}[proposition]{Corollary}
\newtheorem{lemma}[proposition]{Lemma}
\theoremstyle{definition}
\newtheorem{definition}[proposition]{Definition}
\newtheorem{remark}[proposition]{Remark}

\numberwithin{equation}{section}
\def\be{\begin{equation}}
\def\ee{\end{equation}}
\def\bp{\begin{proof}}
\def\ep{\end{proof}}

\def \no#1#2#3 {{\bf #1} (#3), #2.}
\def \eds#1#2#3 {#1, #2, #3.}
\title[Compressible Brinkman-Forcheimer equation]
{Asymptotic regularity and attractors for slightly compressible Brinkman-Forcheimer equations}
\author[V. Kalantarov and S. Zelik]
{ Varga  Kalantarov${}^{1,2}$ and Sergey Zelik${}^{3,4}$}

\address{${}^1$ Department of Mathematics,
\newline\indent Ko{\c c} University, Rumelifeneri Yolu, Sariyer, Istanbul, Turkey
}
\address{${}^2$ Department of General and Applied Mathematics,
\newline\indent Azerbaijan State Oil and Insustry University, Baku, Azerbaijan
}
\email{vkalantarov@ku.edu.tr}
\address{${}^3$
Department of Mathematics,\newline \indent
University of Surrey, GU27XH,
Guildford,  UK}

\address{${}^4$ School of Mathematics and Statistics, Lanzhou University, Lanzhou
\newline\indent 730000,
P.R. China}
\email{s.zelik@surrey.ac.uk}

\begin{document}

\begin{abstract} Slightly compressible Brinkman-Forchheimer equations
 in a bounded 3D domain with Dirichlet boundary conditions are considered. These equations  model fluids
motion in porous media.
The dissipativity of these equations in higher order energy spaces is obtained and regularity
 and smoothing properties of the solutions are studied. In addition, the existence of a global
 and an exponential attractors for these equations in a natural phase space is verified.
\end{abstract}

\thanks{This work is partially supported by  the RSF grant   19-71-30004  as well as  the EPSRC
grant EP/P024920/1}

\maketitle
\tableofcontents

\bigskip
\section{Introduction}
%\vspace{0.5cm}
%\begin{center}
%{\it To the memory of O.A. Ladyzhenskaya and M.I. Vishik}
%\end{center}
%\vspace{0.5cm}

We give a comprehensive study of  slightly compressible Brinkman-Forch\-hei\-mer
equations in the following form:
%$$
\begin{equation}\label{bf1}
\begin{cases}
\Dt u-\Dx u+\Nx p+f(u)= g, \ \ u\big|_{\partial\Omega}=0,\  u\big|_{t=0}=u_0,
 \\ \Dt p+\divv(Du)=0, \ \ p\big|_{t=0}=p_0
\end{cases}
\end{equation}
%$$
in a bounded  domain $\Omega\subset\R^3$ with sufficiently smooth boundary $\p\Om$. Here $u=(u^1(t,x),u^2(t,x),u^3(t,x))$ and
 $p=p(t,x)$ are unknown velocity vector field and pressure respectively, $D$ is a given positive
  self-adjoint matrix, $f$ is a given nonlinearity and
  $g$ is the external force.
\par
Equations of the form \eqref{bf1} arise in the mathematical theory of fluids in porous media and
 are of a big permanent interest from both theoretical and applied points of view,
 see \cite{Aul,Brin,FG,GT,HR,KZ,Lad1,MTT,Mus,R,Str,Tem1,Whit} and references therein. The first equation
  of \eqref{bf1} is usually interpreted as a generalization of the Darcy law:
  $$
  \tau\Dt v-\beta\Dx v+f(v)=-D\Nx p,
  $$
where $D$ is a normalized permeability tensor, $f(v)$ is a Forchheimer nonlinearity which
typically has a form
%$$
\begin{equation}\label{0.fF}
f(v)=\alpha v+\beta (\Cal Cv.v)^{l}v+\gamma\sqrt{(\Cal Cv.v)}v,
\end{equation}
%$$
where $\Cal C$ is another positive self-adjoint matrix and $\alpha,\beta,\gamma$ and $l\ge\frac12$ are
 some constants, see e.g., \cite{Aul} and references therein, $\beta\Dx v$ is a Brinkman term with
 effective viscosity parameter $\beta>0$, see \cite{Brin} and $\tau\Dt v$ is time relaxation
 term which is especially important in the case of non-monotone $f$ or/and presence of the
  inertial term $(v,\Nx)v$ to provide the unique expression of $v$ through $\Nx p$. The second equation
  $$
  \Dt p+\divv(v)=0
  $$
is just a standard slightly compressible approximation of the continuity equation,
see e.g., \cite{Tem1,LST,Donat} and references therein. Making the change of variables $v=Du$, we end up with a system
 of the form \eqref{bf1} with a slightly unusual term $\divv(Du)$.
\par
We also mention that the equations \eqref{bf1} in 2D case naturally arise in the dynamic
theory of tides as a generalization of the classical Laplace tidal equations. In this case,
$u:=(u^1,u^2)$ is the horizontal transport vector
 (the horizontal velocity  averaged over the vertical axis)
  and the scalar  $p$ is a  vertical tidal elevation,
   see, e.g.,  \cite{Go,Ip, Li,MK,Mo} and references therein.
\par
Equations \eqref{bf1} have a non-trivial structure which is interesting also
from purely mathematical point of view. Indeed, in the simplest case $f=g=0$ $D=1$, we may introduce
a new variable $\omega=\operatorname{curl}u$ and reduce the system to the following equations
$$
\Dt\omega-\Dx\omega=0,\ \ \Dt^2 p-\Dx \Dt p-\Dx p=0,
$$
so we see a combination of a heat equation with the so-called strongly damped wave equation. This
system is decoupled in the case of periodic boundary conditions, but in the case of Dirichlet
boundary conditions we have a non-trivial coupling already on the level of linear equations
 through boundary conditions. Thus, in contrast to the incompressible case, one cannot expect instantaneous
 smoothing property for $u$ and $p$, but similarly to damped wave equation, one can expect that
 some components of the solution may have this property, see \cite{KZ09} for more details. Of course, for
  non-zero nonlinearity $f$, we also have coupling through nonlinear terms.
 \par
 Another possibility  is to differentiate the first equation in time and exclude
  the pressure using the second equation. This gives the second order in time equation:
  $$
  \Dt^2 u-\Dx \Dt u+f'(u)\Dt u-\Nx\divv(Du)=0
  $$
which is again a sort of strongly damped wave equation with the nonlinearity of
 Van der Pol type, see e.g., \cite{KZ14} for the regularity and longtime behavior
  of such equations in the scalar case. However, this form of equations \eqref{bf1} is not
   convenient especially for the study of longtime behavior since the operator
    $\Nx \divv(Du)$ is degenerate.
\par
The longtime behavior of solutions to {\it incompressible} Brinkman-Forch\-heimer or
 Brink\-man-Forchheimer-Navier-Stokes equations (often also referred as tamed Na\-vier-Stokes equations)
  is studied in many papers, see \cite{HR,KZ,MTT,WL,YZZ} and references therein. However, the slightly
   compressible case is essentially less understood. To the best of our knowledge, similar problems
    have been considered only in 2D case only for slightly compressible Navier-Stokes equations,
    see \cite{GT} and \cite{FG} for global and exponential attractors respectively, but even
    in this case, Dirichlet's boundary conditions were out of consideration because of the
     problems with obtaining dissipative estimates for the pressure in $H^1(\Omega)$ which are
      caused by "bad" boundary terms in higher energy estimates.
\par
The aim of the present paper is to verify the global well-posedness and dissipativity of the
problem \eqref{bf1} in the initial phase space
$$
(u_0,p_0)\in E:=H^1_0(\Omega)\times \bar L^2(\Omega),\ \ \bar L^2(\Omega):=\{p_0\in L^2(\Omega), \<p_0\>=0\},
$$
where $\<v\>$ is a mean value of the function $v(x)$ as well as in the higher energy space
$$
E^1=E\cap(H^2(\Omega)\times H^1(\Omega))
$$
and to prove the existence of global and exponential attractors for the associated solution semigroup. Note that,
similarly to \cite{GT}, we are unable to verify the dissipativity in $E^1$ using the
 energy-type estimates because of the appearance of "bad" boundary integrals. We overcome this problem
  using the combination of partial instantaneous smoothing property and localization technique inspired
   by \cite{KZ09}. Actually, the localization technique is used here in a bit non-standard way, since it
    is usually applied to verify the higher regularity. In our situation, this higher regularity is more or
    less straightforward and the localization is used in order to get the dissipative estimate only,
    see Appendix \ref{s4} for more details.
\par
Throughout of the paper, we assume that the external force $g\in L^2(\Omega)$ and the
 nonlinearity $f(u)$ has the following form
 %$$
 \begin{equation}\label{0.fstr}
 f(u):=\varphi(|u|^2)u,
 \end{equation}
%$$
where $\varphi\in C^1((0,\infty))$ and satisfies the conditions:
%$$
\begin{equation}\label{0.f}
\begin{cases}
1.\ \  K-Cz^{-1/2}\le\varphi'(z)\le C_1z^{-1/2}(1+z^{3/2}),\\
 2.\ \ -C+\alpha z^{l}\le \varphi(z)\le C(1+z^{l}),\ \ z\in\R_+
\end{cases}
\end{equation}
%$$
for some positive constants $\alpha,K,C,C_1$ and the exponent $l\in(0,2]$.
\par
Clearly these conditions are satisfied for the typical nonlinearity \eqref{0.fF} if $\Cal C=1$
 (or $D\Cal CD=1$
if we take into the account the change of variables mentioned above). The case of a general
self-adjoint positive $\Cal C$ is completely analogous, we only need to take
$\Cal C u.u$ instead of $|u|^2$
 in \eqref{0.fstr}, we assume that $\Cal C=1$ only for simplicity. In contrast to this, the extra assumption
  that $D=1$ somehow oversimplifies the problem since some additional energy type identities hold in this
   particular case, so we prefer to keep a general matrix $D$. We also mention that the exponent
    ${-1/2}$ is fixed in \eqref{0.f} in order to handle the term $\sqrt{(\Cal Cu.u)}\,u$ in \eqref{0.fF}.
     Of course, if $l=\frac12$, we need to assume that $\gamma+\beta>0$ in \eqref{0.fF}
     in order to get dissipativity.
      Analogously, for $l>\frac12$, we need to assume that $\beta>0$.
\par
The paper is organized as follows. In \S1 we derive the basic  dissipative estimate for
problem \eqref{bf1} in the energy phase space $E$, verify the existence and uniqueness of solutions
 and prove, some instantaneous regularization for the $u$ component of the solution. Namely,
 we establish that, starting from $(u(0),p(0))\in E$, at the next time moment $t$, we will
 have $\Dt u(t)\in L^2(\Omega)$ and $\Nx u(t)\in L^2(\Omega)$. This regularization allows us,
 similarly to the case of strongly damped wave equations, to truncate system \eqref{bf1} and
  reduce the analysis to simpler equations:
  %$$
  \begin{equation}\label{0.trunc}
  \Dt p+\divv(Du)=0,\ \ p\big|_{t=0}=p_0,\ \ -\Dx u+\Nx p+f(u)=g(t),
  \end{equation}
%$$
where $g\in L^\infty(\R_+,L^2(\Omega))$ is new given external force (of course, the relation of
this system to the initial equations \eqref{bf1} is given by $g(t):=g-\Dt u(t)$).
\par
The detailed analysis of this truncated system is presented in \S\ref{s2}. In particular, we
 prove there that this system is well-posed and dissipative in higher energy space
 $p\in \bar L^2(\Omega)\cap H^1(\Omega)$ and also establish the exponential smoothing
  property for this system, namely, we check that the ball in the space $H^1(\Omega)$
  attracts exponentially fast the trajectories $p(t)$ of \eqref{0.trunc} starting from
   bounded sets of $\bar L^2(\Omega)$. Returning back to the full system \eqref{bf1}, we
   establish after that its well-posedness and dissipativity in higher energy space $E^1$ as well as the fact that
   the proper ball in $E^1$ is an exponentially attracting set for the solutions of \eqref{bf1} starting from $E$. This
    fact, in turn, is crucial for our study of global  and exponential attractors.
\par
     Note also that the analysis
     presented in this section is heavily based on the study of linear problem \eqref{0.trunc} (which
     corresponds to $f=0$) presented in Appendix \ref{s4} and, in particular, on the dissipativity
     of this linear problem in higher energy space $H^1(\Omega)$. This dissipativity is proved using the
     localization technique and is of independent interest.
     \par
    In \S\ref{s3} we verify the existence of a global and  exponential attractors for
    the solution semigroup associated with problem \eqref{bf1}. These results are more or less
    standard corollaries of the asymptotic regularity and exponential attraction proved in \S\ref{s2},
    see \cite{BV,cv,EFNT,EMZ05,MZ,Tem} for more details.
\par
    Finally, in \S\ref{s7}, we also consider briefly some generalizations
    of the proved results, including  the case of the extra convective terms in  the initial
     Brink\-man-Forch\-hei\-mer equation
     and discuss some open problems for further research.

\section{Well-posedness, dissipativity and partial smoothing}\label{s1}

In this section, we verify the global well-posedness and dissipativity of slightly compressible
 Brinkman-Forchheimer equations:
%$$
\begin{equation}\label{1.main}
\begin{cases}
\Dt u-\Dx u+\Nx p+f(u)=g,\ \ u\big|_{\partial\Omega}=0, \ \ u\big|_{t=0}=u_0,\\
 \Dt p+\divv(Du)=0,\ \  p\big|_{t=0}=p_0
\end{cases}
\end{equation}
%$$
in the energy space $(u_0,p_0)\in E$ as well as establish some partial smoothing results for the solutions of this system
 which are crucial for what follows.
We start with the basic a priori estimate in the phase space $E$.

\begin{theorem}\label{Th1.E-dis} Let $g\in L^2(\Omega)$, $D=D^*>0$ and the nonlinearity $f$
satisfy \eqref{0.fstr} and \eqref{0.f}. Let also $(u(t),p(t))$ be a sufficiently
smooth solution of \eqref{1.main}. Then, the following estimate holds:
%$$
\begin{multline}\label{1.E-dis}
\|(u,p)(t)\|_{E}^2+\int_t^{t+1}(\|\Nx u(s)\|^2_{L^2}+(|f(u(s))\cdot Du(s)|,1))\,ds \le \\\le Q(\|(u,p)(0)\|^2_{E})e^{-\alpha t}+Q(\|g\|^2_{L^2}),
\end{multline}
%$$
for some monotone function $Q$ and positive constant $\alpha$ independent on $u$ and $t$.
\end{theorem}
\begin{proof} We multiply the first equation of \eqref{1.main} by $Du$ and integrate over $\Om$. Then,
integrating by parts and using the second equation, we arrive at
%$$
\begin{equation}\label{1.eq}
\frac12\frac d{dt}\(\|u\|^2_{L^2_D}+\|p\|^2_{\bar L^2}\)+\|\Nx u\|^2_{L^2_D}+(f(u),Du)=(g,Du),
\end{equation}
%$$
where $\|u\|^2_{L^2_D}:=\int_\Omega Du(x).u(x)\,dx$. Here and below $\xi.\eta$ stands for
the standard dot product of vectors $\xi,\eta\in\R^3$.
\par
This energy identity is still not enough to get the dissipative estimate since it does not contain the term
 $\|p\|^2_{L^2}$ without time differentiation. To get this term we use the so-called Bogovski operator:
 %$$
\begin{equation}
 \mathfrak B:\bar L^2(\Omega)\to H^1_0(\Omega),\
 \bar L^2(\Omega):=\{p\in L^2(\Omega),\ \<p\>=0\},\ \ \divv \mathfrak Bp=p.
\end{equation}
It is well-know that such an operator exists as a linear continuous operator if $\Omega$
is smooth enough, see e.g.,\cite{S}. Multiplying the first equation of \eqref{1.main} by $\mathfrak Bp$,
integrating with respect to  $x$ and using the second equation, we get
%$$
\begin{multline}\label{1.eq1}
\frac d{dt}(u,\mathfrak Bp)+\|p\|^2_{\bar L^2}=-(p,\mathfrak B\divv(Du))-\\-
(\Nx u,\Nx \mathfrak Bp)-(f(u),\mathfrak Bp)+(g,\mathfrak Bp).
\end{multline}
%$$
Multiplying \eqref{1.eq1} by a small $\eb>0$ and taking a sum with equation \eqref{1.eq}, after using
 the H\"older inequality and Sobolev embedding $H^1\subset L^6$, we  get
 %$$
 \begin{multline}
\frac d{dt}\(\|u\|^2_{L^2_D}+\|p\|^2_{\bar L^2}+2\eb(u,\mathfrak Bp)\)+\\+
\|\Nx u\|^2_{L^2_D}+\eb\|p\|^2_{\bar L^2}+(f(u),Du)
\le\eb\|f(u)\|_{L^{6/5}}\|p\|_{\bar L^2}+C(\|g\|^2_{L^2}+1).
 \end{multline}
 %$$
 Using our assumptions \eqref{0.f}  on functions $f$ and $\varphi$, it is not difficult to verify that
 %$$
 \begin{equation}
|f(u)|^{6/5}\le C(|f(u).Du|+1).
 \end{equation}
 This gives us the following differential inequality:
 %$$
 \begin{equation}\label{1.Gp}
\frac d{dt}\Cal E_\eb(u,p)+\eb \Cal E_\eb(u,p)\le C\eb^6\Cal E_{\eb}^3(u,p)+C(\|g\|^2_{L^2}+1),
 \end{equation}
 %$$
 where
 %$$
\begin{equation}
 \Cal E_\eb(u,p):=\|u\|^2_{L^2_D}+\|p\|^2_{\bar L^2}+2\eb(u,\mathfrak Bp).
\end{equation}
%$$
Moreover, for sufficiently small $\eb>0$, we have
$$
\frac12\|(u,p)\|_{E}^2\le\Cal E_\eb(u,p)\le\frac32\|(u,p)\|^2_{E}.
$$
Thus, applying the Gronwall lemma with a parameter to \eqref{1.Gp}, see \cite{GPZ,Pat} and also
 \cite{Z07} for details, we end up with
 the desired estimate \eqref{1.E-dis} and finishes the proof of the theorem.
\end{proof}
At the next step we define a weak energy solution for problem \eqref{1.main}.

\begin{definition}\label{Def1.sol} A pair of function $(u,p)\in C_w(0,T;E)$ is a weak energy
 solution of problem \eqref{1.main} if, in addition,
 %$$
 \begin{equation}
u\in L^2(0,T;H^1_0(\Omega))\cap L^{2(l+2)}(0,T;L^{2(l+2)}(\Omega))
 \end{equation}
 %$$
 and the equations \eqref{1.main} are satisfied in the sense of distributions.
\end{definition}
\begin{remark}\label{Rem1.cont} From the first equation \eqref{1.main}, we see that
%$$
\begin{multline}
\Dt u\in L^2(0,T;H^{-1}(\Omega))+L^{l'}(0,T;L^{l'}(\Omega))=\\=
[L^2(0,T;H^1_0(\Omega))\cap L^{2(l+1)}(0,T;L^{2(l+1)}(\Omega))]^*,
\end{multline}
%$$
where $\frac1{l'}+\frac1{2(l+1)}=1$.
Thus, multiplication of the first equation of \eqref{1.main} by $u$ or $\mathfrak Bp$ is
 justified on the level of weak
 energy solutions and we have, in addition,
that $u\in C(0,T;L^2(\Omega))$, see e.g., \cite{cv} for the details. The situation with
 $p$-component is even simpler
 since we have
 $$
 \Dt p\in L^2(0,T;\bar L^2(\Omega))
 $$
 and multiplication on $p$ is allowed. So, we also have that $p\in C(0,T;\bar L^2(\Omega))$.
 In particular, all manipulations done for the derivation of the key estimate \eqref{1.Gp} are
  actually justified for weak energy solutions, so all such solutions satisfy the dissipative
   estimate \eqref{1.E-dis}.
\end{remark}
We now turn to the uniqueness.
\begin{theorem}\label{Th1.un} Suppose that  the  assumptions of Theorem \ref{Th1.E-dis} are satisfied and let $(u_1(t),p_1(t))$
and $(u_2(t),p_2(t))$ be two weak energy
solutions of problem \eqref{1.main}. Then, the following estimate holds:
%$$
\begin{multline}\label{1.lip}
\|(u_1(t)-u_2(t),p_1(t)-p_2(t))\|_E^2\le\\\le Ce^{Kt}\|(u_1(0)-u_2(0),p_1(0)-p_2(0))\|_E^2,
\end{multline}
%$$
where the constants $C$ and $K$ depend only on $f$ and $D$.
\end{theorem}
\begin{proof} We first note that it suffices to verify \eqref{1.lip} for $t\le T$ for some small,
but positive $T$. Then, to get the general estimate, it will be enough to iterate \eqref{1.main}.
Let $\bar u(t):=u_1(t)-u_2(t)$ and $\bar p(t):=p_1(t)-p_2(t)$. Then, these functions solve
%$$
\begin{equation}\label{1.dif}
\Dt\bar u-\Dx\bar u+\Nx\bar p+[f(u_1)-f(u_2)]=0,\ \ \Dt\bar p+\divv(D\bar u)=0.
\end{equation}
%$$
Integrating the second equation, we get
%$$
\begin{equation}\label{1.p}
\bar p(t)=\bar p(0)-\int_0^t\divv(D\bar u(s))\,ds.
\end{equation}
%$$
Multiplying now the first equation by $\bar u(t)$ and using that,
due to assumptions \eqref{0.fstr} and \eqref{0.f}, $f'(u)\ge-L$, after the standard
 transformations, we end up with
 %$$
 \begin{multline}
\frac d{dt}\|\bar u(t)\|^2_{L^2}+\|\Nx\bar  u\|^2_{L^2}+\\+2\(\int_0^t\divv(D\bar u(s))\,ds,\divv\bar u(t)\)\le
 C\|\bar p(0)\|^2_{\bar L^2}+2L\|\bar u(t)\|^2_{L^2}.
 \end{multline}
 %$$
Assuming that $T$ is small enough, we estimate
%$$
\begin{multline}
\Big|2\(\int_0^t\divv(D\bar u(s))\,ds,\divv\bar u(t)\)\Big|\le \\
\le \int_0^t\|\divv(D\bar u(s))\|^2_{L^2}\,ds+
T\|\divv(\bar u(t))\|^2_{L^2}\le \\\le C\int_0^t\|\Nx\bar u(s)\|^2_{L^2}\,ds+\frac12\|\Nx\bar u(t)\|^2_{L^2}
\end{multline}
%$$
and, therefore,
%$$
 \begin{multline}
\frac d{dt}\|\bar u(t)\|^2_{L^2}+\frac12\|\Nx\bar  u\|^2_{L^2}-\\-C\int_0^t\|\Nx\bar u(s)\|^2_{L^2}\,ds\le
 C\|\bar p(0)\|^2_{\bar L^2}+2L\|\bar u(t)\|^2_{L^2}.
 \end{multline}
 %$$
Integrating this inequality in time, we end up with
%$$
\begin{multline}
\|\bar u(t)\|^2_{L^2}+\int_0^t\(\frac12-C(t-s)\)\|\Nx \bar u(s)\|^2_{L^2}\,ds \le
\\
\le
 C\|(\bar u(0),\bar p(0))\|^2_{E}+2L\int_0^t\|\bar u(s)\|^2_{L^2}\,ds.
\end{multline}
%$$
Fixing now $T$ small enough that the integral in the left-hand side is positive and applying the
 Gronwall inequality,
we finally arrive at
%$$
\begin{equation}
\|\bar u(t)\|^2_{L^2}+\int_0^t\|\Nx\bar u(s)\|^2_{L^2}\,ds\le K\|(\bar u(0),\bar p(0))\|^2_{L^2},\ \ t\le T.
\end{equation}
%$$
The corresponding estimate for the $p$-component follows now from \eqref{1.p}. Thus, the
estimate \eqref{1.lip} is verified and the theorem is proved.
\end{proof}
\begin{corollary}\label{Cor1.sem} Let the assumptions of Theorem \ref{Th1.E-dis} hold.
 Then, equations \eqref{1.main} generate a dissipative
globally Lipschitz continuous semigroup $S(t)$ in the phase space $E$:
%$$
\begin{equation}\label{1.sem}
S(t)(u_0,p_0):=(u(t),p(t)),
\end{equation}
%$$
where $(u(t),p(t))$ is a unique  energy solution of \eqref{1.main}
 with the initial data $(u_0,p_0)\in E$.
\end{corollary}
\begin{proof} According to theorems \ref{Th1.E-dis} and \ref{Th1.un}, we only need to verify the
 existence of a weak solution. This can be done in many standard ways, one of the is to use vanishing
  viscosity method. Namely, we may approximate \eqref{1.main} by a family of parabolic equations:
  $$
  \Dt u-\Dx u+\Nx p=g,\ \Dt p+\divv(Du)=\nu\Dx p,\ u\big|_{\partial\Omega}=0,\
   \partial_np\big|_{\partial\Omega}=0,
  $$
  where $\nu>0$ is a small parameter. The solution of this parabolic problem can be obtained using e.g.,
  the Galerkin approximations and the passage to the limit as $\nu\to0$ is also straightforward since
   the analogue of \eqref{1.eq} gives the necessary uniform with respect to $\nu\to0$ estimates (although they are non-dissipative,
   this is not important for the existence of a solution on a finite time interval).
    So, we omit the details here.
\end{proof}
By the analogy with strongly damped wave equation (see \cite{PV06,KZ09} and references therein),
one may expect that \eqref{1.main} partially possesses
instantaneous smoothing property. The next results shows
that such a smoothing indeed holds.

\begin{theorem}\label{Th1.sm-ins} Let the  assumptions of Theorem \ref{Th1.E-dis} hold and let $(u,p)$ be a weak
 energy solution of  problem \eqref{1.main}.
Then the following partial smoothing property holds:
%$$
\begin{multline}\label{1.sm}
t\|\Nx u(t)\|^2_{L^2}+t^2\|\Dt u(t)\|^2_{L^2}+t\|\Dt p(t)\|^2_{L^2}+\\+
\int_0^{t}s^2\|\Nx\Dt u(s)\|^2_{L^2}\,ds\le C(1+\|g\|^2_{L^2}+\|(u(0),p(0))\|^2_{E}),
\end{multline}
%$$
where $t\in[0,1]$ and the constant $C$ is independent of $t$ and $u$.
\end{theorem}
\begin{proof} Let us first multiply the first equation of \eqref{1.main} by $t\Dt u$ and integrate over $\Om$. Then,
using the gradient structure of nonlinearity $f$, we arrive at
%$$
\begin{multline}\label{1.dt}
\frac d{dt}\(\frac t2\|\Nx u\|^2_{L^2}+t(F(u),1)-t(p,\divv u)\)+t\|\Dt u\|^2_{L^2}=\\=
-(p,\divv u)+\frac12\|\Nx u\|^2_{L^2}+(F(u),1)+t(\divv(Du),\divv u)+t(g,\Dt u).
\end{multline}
%$$
where $F(u):=\frac12\int_0^{|u|^2}\varphi(z)\,dz$.
Integrating this identity in time and using  the estimate \eqref{1.E-dis}, we arrive at the following
smoothing property:
%$$
\begin{multline}\label{1.sm1}
t\|\Nx u(t)\|^2_{L^2}+t\|u(t)\|_{L^{2(l+1)}}^{2(l+1)}+t\|\Dt p(t)\|^2_{L^2}+\\
+\int_0^ts\|\Dt u(s)\|^2_{L^2}
\le C\(\|(u(0),p(0))\|^2_{E}+1+\|g\|^2_{L^2}\),
\end{multline}
%$$
where $t\in[0,1]$ and $C$ is independent of $u$, $p$ and $t$.
\par
Let us now differentiate equations \eqref{1.main} in time and denote
 $v:=\Dt u$ and $q=\Dt p$. Then, we end up with the following equations
%$$
\begin{equation}\label{1.mdif}
\Dt v-\Dx v+\Nx q+f'(u)v=0,\ \ \Dt q+\divv(Dv)=0.
\end{equation}
%$$
Multiplying the first equation of \eqref{1.mdif} by $t^2 v$ and integrating over $\Om$, we get
%$$
\begin{multline}\label{1.dtv}
\frac12\frac d{dt}\left(t^2\|v(t)\|^2_{L^2}\right)+t^2\|\Nx v\|^2_{L^2}+
\\+t^2(f'(u)v,v)= t^2(\Dt p,\divv v)+
t\|\Dt u\|^2_{L^2}.
\end{multline}
%$$
Integrating this equality in time and using \eqref{1.sm1} together with the assumption $f'(u)\ge-L$,
we get the desired smoothing property in the form
%$$
\begin{multline}
t^2\|\Dt u(t)\|^2_{L^2}+\int_0^ts^2\|\Nx \Dt u(s)\|^2_{L^2}\,ds \le \\\le
 C(1+\|g\|^2_{L^2}+\|(u(0),p(0))\|_{E}^2),
\end{multline}
%$$
where $t\in[0,1]$ and finish the proof of the theorem.
\end{proof}
Combining smoothing estimate \eqref{1.sm} with the dissipative estimate
 \eqref{1.E-dis}, we get the following result.

\begin{corollary}\label{Cor1.smdt} Let the  assumptions of Theorem \ref{Th1.E-dis} hold and let $(u,p)$ be a weak
energy solution of problem \eqref{1.main}. Then, we have
 the following dissipative
 estimate for higher norms:
%$$
\begin{multline}\label{1.ins-sm}
\|\Nx u(t)\|^2_{L^2}+\|\Dt u(t)\|^2_{L^2}+\int_t^{t+1}\|\Nx\Dt u(s)\|^2_{L^2} \le \\
\le
\frac{t^2+1}{t^2}\(Q(\|(u(0),p(0))\|_E)e^{-\alpha t}+Q(\|g\|_{L^2})\),
\end{multline}
%$$
where the positive constant $\alpha$ and monotone function $Q$ are independent of $t$, $u$ and $p$.
\end{corollary}
This estimate, in turn, allows us (analogously to the case of strongly damped
 wave equations, see \cite{PV06,KZ09}) to reduce the study of the asymptotic smoothness
  for solutions to the following truncated auxiliary problem
  %$$
  \begin{equation}\label{1.trunc}
  \begin{cases}
  -\Dx u+\Nx p+f(u)=g(t),\ u\big|_{\partial\Omega}=0,\\
  \Dt p+\divv(Du)=0,\  p\big|_{t=0}=p_0,
\end{cases}
  \end{equation}
  %$$
where the external force $g(t)=g-\Dt u(t)$ satisfies the estimate
%$$
\begin{equation}\label{1.g}
\|g\|_{L^\infty(\R_+,L^2(\Omega))}\le C
\end{equation}
%$$
which will be studied in the next sections. We also mention here that, in order to restore the
 $u$-component of a solution $(p,u)$ of this problem in a unique way by the $p$-component,
 we need to assume in addition that
 %$$
 \begin{equation}\label{1.fmon}
 f'(u)\ge0.
 \end{equation}
 %$$
This assumption however, is not restrictive since, in a general case, the extra term $Lu$ can be added to
the nonlinearity and also to the external force $g(t)$ and the $L^2(\Omega)$-norm of this term
 is under the control.

\section{Asymptotic regularity}\label{s2}

In this section, we study the asymptotic smoothing for the truncated system \eqref{1.trunc} which is also
of independent interest. We will mainly
 concentrate here on the case of critical quintic growth rate of the nonlinearity ($f(u)\sim u^5$).
 The subcritical case is essentially simpler since the standard linear splitting of the solution
  semigroup on a contracting and compact components works. In contrast to this, we need a {\it nonlinear}
  splitting in the critical case. Moreover, due to specific structure of our problem, we need a
  combination of different decompositions.  We start with the following splitting
  $$
  p=q+r, \ \ u=v+w,
  $$
  where
  %$$
  \begin{equation}\label{3.comp}
  \begin{cases}
  \Dt q+\divv(Dv)=0,\ \ q\big|_{t=0}=p\big|_{t=0},\\
   -\Dx v+\Nx q+f(v)+Lv=0,\ \ v\big|_{\partial\Omega}=0
\end{cases}
  \end{equation}
  %$$
  and
  %$$
  \begin{equation}\label{3.contr}
\begin{cases}
\Dt r+\divv(Dr)=0,\ \ r\big|_{t=0}=0,\\
-\Dx w+\Nx r+[f(u)-f(v)]=Lv+g(t),\ \ w\big|_{\partial\Omega}=0.
\end{cases}
  \end{equation}
  %$$
According to the results of previous section, we may assume without loss of generality that
 $p(0)$ belongs to the absorbing ball in $\bar L^2(\Omega)$. Then, from the analogues of dissipative
 estimates for equation \eqref{3.comp}, we conclude that
 %$$
 \begin{equation}\label{3.R}
\|p(t)\|_{L^2}+\|q(t)\|_{\bar L^2}+\|r(t)\|_{\bar L^2}+\|u(t)\|_{H^1}+\|v(t)\|_{H^1}+\|w(t)\|_{H^1}\le R
 \end{equation}
 %$$
 for all $t\ge0$. We start with the contracting part $(q,v)$.

 \begin{proposition}\label{Prop3.contr} Let the function $f$ satisfy \eqref{0.f},
  \eqref{0.fstr} and \eqref{1.fmon}, $D=D^*>0$ and estimates \eqref{3.R} and \eqref{1.g} hold.
   Then, there exists $L=L(R)$
 such that the solution $r(t)$ of the problem \eqref{3.contr} satisfies the estimate:
 %$$
 \begin{equation}
\|q(t)\|^2_{\bar L^2}+\|v(t)\|^2_{H^1}\le Ce^{-\alpha t}\|p(0)\|^2_{\bar L^2},
 \end{equation}
 %$$
 where positive constants $C$ and $\alpha$ are independent of $t$, $u$ and $p$.
 \end{proposition}
\begin{proof} We fix $L>0$ in such a way that
$$
f(v).Dv+Lv.Dv\ge0,\ v\in\R^3
$$
(it is possible to do so since $f(0)=0$ and $f(v).Dv\ge-C$). Then, multiplying the first and second equations
 of \eqref{3.comp} by $q$ and $Dv$ respectively and integrating over $\Om$, we end up with
 %$$
\begin{equation}
 \frac12\frac d{dt}\|q\|^2_{\bar L^2}+\|\Nx v\|^2_{L^2_D}\le 0.
\end{equation}
%$$
Multiplying now the second equation of \eqref{3.comp} by $\mathfrak Bq$ and using the inequality
$$
\|f(v)\|_{H^{-1}}\le C(1+\|v\|_{H^1}^4)\|v\|_{H^1}\le C_R\|\Nx v\|_{L^2_D},
$$
we infer that
$$
\|q\|_{\bar L^2}^2\le C'_R\|\Nx v\|_{L^2_D}^2
$$
and, therefore,
$$
\frac12\frac d{dt}\|q\|^2_{\bar L^2}+\alpha_R\|q\|^2_{\bar L^2}\le 0,
$$
for some positive $\alpha_R$ depending only on $R$. Applying the Gronwall inequality, we arrive
 at the desired estimate for $q$:
 $$
 \|q(t)\|^2_{\bar L^2}\le e^{-\alpha_R t}\|p(0)\|^2_{\bar L^2}.
 $$
 To get the desired estimate for $\|v\|^2_{H^1}$, it remains to note that multiplication of
 the second equation of \eqref{3.comp} by $Dv$ gives
 $$
 \|\Nx v(t)\|^2_{L^2_D}\le C\|q(t)\|^2_{\bar L^2}.
 $$
 Thus, the proposition is proved.
\end{proof}
We now turn to the smooth part $(w(t),r(t))$ of the solution generated by the problem \eqref{3.contr}.
 At the first step, we derive {\it exponentially growing} estimate for this part in higher norms
 which will be improved later.

 \begin{proposition}\label{Prop3.gr} Let the  assumptions of Proposition \ref{Prop3.contr}
 hold and let $\delta\in(0,\frac12)$.
 Then, the following estimate for the
 solution $(w(t),r(t))$ of \eqref{3.contr} is valid:
%$$
\begin{equation}\label{3.gr}
\|r(t)\|^2_{H^\delta}+\|w(t)\|^2_{H^{1+\delta}}\le Ce^{Kt},
\end{equation}
%$$
where $K>0$ and the constant $C$ depends on $g$ (through assumption \eqref{1.g})
 and $R$, but is independent of $t$, $p$ and $u$.
 \end{proposition}
\begin{proof} To verify this estimate we need the following standard lemma.

\begin{lemma}\label{Lem3.w} Let $a(x)\ge0$ be a symmetric measurable matrix and
the function $w\in H^1_0(\Omega)\cap L^2_a(\Omega)$ be a solution of the following problem:
%$$
\begin{equation}\label{3.ell}
-\Dx w+a(x)w=\Nx r+g,\ \ w\big|_{\partial\Omega}=0,
\end{equation}
%$$
where $L^2_a(\Omega)$ is a weighted Lebesgue space determined by the semi-norm
$$
\|w\|_{L^2_a}^2:=\int_{\Omega}a(x)w(x)\cdot w(x)\,dx<\infty,
$$
$r\in \bar H^\delta(\Omega):=H^\delta(\Omega)\cap \bar L^2(\Omega)$ for some $\delta\in(0,\frac12)$, and $g\in L^2(\Omega)$. Then,
 the following estimate holds:
%$$
\begin{equation}\label{w-est}
\|w\|_{L^s}\le C(\|r\|_{\bar H^\delta}+\|g\|_{L^2}),
\end{equation}
%$$
where the constant $C$ is independent of $a$, $g$, $w$ and $r$ and $s=\frac6{1-2\delta}$ is the Sobolev embedding
 exponent for $H^{1+\delta}\subset L^s$.
\end{lemma}
\begin{proof}[Proof of the lemma] Since $g$ is more regular than $\Nx r$, it suffices
 to verify the estimate for $g=0$ only. We give below only the formal derivation of  \eqref{w-est} which can
 be justified by standard approximation arguments. To this end,  we multiply equation \eqref{3.ell} by $w|w|^n$, where
  the exponent $n$ will be fixed later and integrate over $\Om$. This gives
  $$
  (|\Nx w|^2,|w|^n)+\|\Nx(|w|^{\frac{n+2}2})\|^2_{L^2}\le C(|r|(|\Nx w||w|^{n/2}),|w|^{n/2}).
  $$
Using the proper H\"older inequality together with Sobolev embeddings, we get
$$
(|\Nx w|^2,|w|^n)+\|w\|_{L^{3(n+2)}}^{n+2}\le
 C\|r\|_{\bar H^\delta}(|\Nx w|^2,|w|^n)^{1/2}\|w\|^{n/2}_{L^{mn/2}},
$$
where $\frac12-\frac\delta3+\frac12+\frac1m=1$, i.e., $m=\frac3\delta$. Therefore, we have
$$
\|w\|_{L^{3(n+2)}}^{n+2}\le C\|r\|_{\bar H^\delta}^{n+2}+\frac12\|w\|^{n+2}_{L^{\frac{3n}{2\delta}}}.
$$
Fixing now $n$ in such a way that $3(n+2)=\frac{3n}{2\delta}$, we see that $3(n+2)=s$ and the last
 estimate finishes the proof of the lemma.
\end{proof}
We now return to the proof of the proposition. First, applying  the lemma to the second
 equation of \eqref{3.contr} with
$$
a(x):=\int_0^1f'(\kappa u(x)+(1-\kappa) v(x))\,d\kappa\ge0,
$$
we end up with
%$$
\begin{equation}
\|w(t)\|_{L^s}\le C(\|r(t)\|_{\bar H^\delta}+\|g(t)\|_{L^2}+L\|v(t)\|_{L^2})\le
C\(\|r(t)\|_{\bar H^\delta}+1\).
\end{equation}
%$$
Second, using the growth restriction on $f$ and Sobolev embedding theorems, it is not difficult to see that
%$$
\begin{equation}\label{3.ffrac}
\|f(u)-f(v)\|_{H^{-1+\delta}}\le C(1+\|u\|_{H^1}^4+\|v\|^4_{H^1})\|u-v\|_{L^s}.
\end{equation}
%$$
Therefore,
%$$
\begin{equation}\label{3.ffrac1}
\|f(u(t))-f(v(t))\|_{H^{-1+\delta}}\le C_R\|w(t)\|_{L^s}\le C_R\(\|r(t)\|_{\bar H^\delta}+1\).
\end{equation}
%$$
Third, we multiply the second equation of \eqref{3.contr} by $(-\Dx)^\delta w$ and integrate over $\Om$. This gives
%$$
\begin{multline}
\|w\|^2_{H^{1+\delta}}=((-\Dx)^{-1+\delta/2}\Nx r,(-\Dx)^{1+\delta/2}w)+\\+
(-\Dx)^{-1+\delta/2}[f(u)-f(v)],(-\Dx)^{1+\delta/2}w)-(Lv(t)+g(t),(-\Dx)^{\delta}w)
\end{multline}
%$$
and therefore
%$$
\begin{equation}\label{3.wr}
\|w\|_{H^{1+\delta}}\le C(\|f(u)-f(v)\|_{H^{-1+\delta}}+\|r\|_{\bar H^{\delta}}+1)\le
 C_R\(1+\|r\|_{\bar H^\delta}\).
\end{equation}
%$$
Finally, from the first equation of \eqref{3.contr}, we get
%$$
\begin{multline}
\|r(t)\|_{\bar H^\delta}\le\int_0^t\|\divv(Dw(\tau))\|_{H^\delta}\,d\tau \le \\\le
 C\int_0^t\|w(\tau))\|_{H^{1+\delta}}\,d\tau\le C_R\int_0^t(\|r(\tau)\|_{\bar H^\delta}+1)\,d\tau
\end{multline}
%$$
and the Gronwall inequality finishes the proof of the proposition.
\end{proof}
At the next step, we split following \cite{Z04} (see also \cite{MSSZ} for some improvements) the
 solution $u(t)$ of \eqref{1.trunc} on uniformly small ($\bar u(t)$  and smooth
 ($\tilde u(t)$) parts.

\begin{proposition}\label{Prop3.split} Let $\beta>0$ be arbitrary and $\delta\in(0,\frac12)$. Let also
 $(p(t),u(t))$ be a solution of \eqref{1.trunc}
satisfying \eqref{3.R}. Then, there exists $T=T_\delta$ such that
 the function $u(t)$ can be split in a sum
%$$
\begin{equation}\label{3.split}
u(t)=\bar u(t)+\tilde u(t),
\end{equation}
%$$
where for every $t\ge T$
%$$
\begin{equation}\label{3.good}
\|\bar u(t)\|_{H^1}\le\beta,\ \ \|\tilde u(t)\|_{H^{1+\delta}}\le C_\beta
\end{equation}
%$$
and the constant $C_\beta$ depends only on $\beta$, $\delta$ and $R$.
\end{proposition}
\begin{proof} This splitting is an almost immediate corollary of the proved Propositions \ref{Prop3.contr}
and \ref{Prop3.gr}. Indeed, let us fix $T=T_\beta$ from the equation
$$
Ce^{-\alpha T}R^2=\beta^2,
$$
where all of the constants are the same as in Proposition \ref{Prop3.contr}. Then, for the $v$-component
 of the solution $u$, we will have the estimate
 $$
 \|v(t)\|_{H^1}\le\beta,\ \ t\ge T.
 $$
 Moreover, if we fix $C_\beta$ from $Ce^{2KT}=C_\beta^2$ where the constants
  are the same as in \eqref{3.gr}, we get
 $$
 \|w(t)\|_{H^{1+\delta}}\le C_\beta
 $$
 if $t\le 2T$. Thus, functions $v(t)$ and $w(t)$ give the desired splitting of $u(t)$ for $t\in[T,2T]$.
 \par
 To construct the desired splitting for all $t\ge T$, we define functions $(q_n(t),v_n(t))$
 and $(r_n(t),w_n(t))$ for all $n\in\Bbb N$ as solutions of \eqref{3.comp} and \eqref{3.contr} respectively,
 but starting from $t=T(n-1)$ with the initial conditions
 $$
 q_n\big|_{t=T(n-1)}=0,\ \ r_n\big|_{t=T(n-1)}=p\big|_{t=T(n-1)}.
 $$
 Then, arguing analogously, we see that $u(t)=v_n(t)+w_n(t)$ gives the required splitting  on the
 interval $t\in[Tn,T(n+1)]$. Finally, to get the desired splitting for all $t\ge T$, we define $\bar u$ and
  $\tilde u(t)$ as hybrid piece-wise continuous functions:
  $$
  \bar u(t)=v_n(t),\ t\in[Tn,T(n+1)),\ \ \tilde u(t)=w_n(t),\ \ t\in[Tn,T(n+1)),\ \ n\in\Bbb Z.
  $$
  This finishes the proof of the proposition.
\end{proof}
We are now ready to refine Proposition \ref{Prop3.gr} and get the dissipative estimate for $(r(t),w(t))$.

\begin{proposition}\label{Prop3.dis} Let the assumptions of Proposition \ref{Prop3.gr} hold.
Then the solution $(r(t),w(t))$ of problem \eqref{3.contr} satisfies the estimate
%$$
\begin{equation}\label{3.notgr}
\|r(t)\|_{\bar H^{\delta}}+\|w(t)\|_{H^{1+\delta}}\le C,
\end{equation}
%$$
where the constant $C$ depends on $R$, but is independent of $u$, $p$ and $t$.
\end{proposition}
\begin{proof} Without loss of generality, we may assume that estimates \eqref{3.good} hold
for $t\ge0$. The general case is reduced to this particular one by the proper time shift. The only
difference is that we need to put non-zero initial data for $r(t)$. Since the $H^\delta$ norm of $r(t)$ on
 the interval $t\in[0,T]$ can be controlled by \eqref{3.gr}, we just need to assume that
 %$$
 \begin{equation}
r\big|_{t=0}=r_0,\ \ \|r_0\|_{H^\delta}\le C_\beta.
 \end{equation}
 %$$
 This also gives that
 %$$
 \begin{equation}\label{3.vgood}
 \|v(t)\|_{H^1}\le \beta,\ \ t\ge0.
 \end{equation}
 %$$
Moreover, again without loss of generality, we may assume that $f'(0)=0$. In a general case, the term
 $f'(0)w(t)$ is lower order and can be treated as a part of $g(t)$.
\par
The idea of the proof is to refine estimate \eqref{3.ffrac1}
using the result of Proposition \ref{Prop3.split}. First, we refine \eqref{3.ffrac} using the fact
 that $f'(0)=0$, namely, this assumption gives us that
 %$$
 \begin{equation}\label{3.ffrac3}
 \|f(u)-f(v)\|_{H^{-1+\delta}}\le C(\|u\|_{H^1}+\|v\|_{H^1})(1+\|u\|_{H^1}^3+\|v\|_{H^1}^3)\|u-v\|_{L^s}
 \end{equation}
 %$$
 for some constant $C$ depending only on $f$. Second, we write
 $$
 f(u)-f(v)=[f(\bar u+\tilde u)-f(\bar u)]+[f(\bar u)-f(v)]
 $$
 and apply \eqref{3.ffrac3} to both terms on the right-hand side. Indeed,
 since $H^{1+\delta}\subset L^s$ and the function $\tilde u$ is bounded in $H^{1+\delta}$, we have
 $$
 \|f(\bar u+\tilde u)-f(\bar u)\|_{H^{-1+\delta}}\le
 C(1+\|u\|_{H^1}^4+\|\bar u\|_{H^1}^4)\|\tilde u\|_{H^{1+\delta}}\le C_1
$$
for some $C_1>0$ which depends on $\beta$ and $R$. Applying estimate
\eqref{3.ffrac3} to the second term and using inequalities \eqref{3.good} and \eqref{3.vgood}, we get
$$
\|f(\bar u)-f(v)\|_{H^{-1+\delta}}\le C\beta\|\bar u-v\|_{L^s}
$$
and using that
$$
\|\bar u-v\|_{L^s}=\|\tilde u-w\|_{L^s}\le \|w\|_{L^s}+C\|\tilde u\|_{H^{1+\delta}}\le \|w\|_{L^s}+C
$$
we get
%$$
\begin{equation}\label{3.ffrac4}
\|f(u)-f(v)\|_{H^{-1+\delta}}\le C\beta\|w\|_{L^s}+C_\beta,
\end{equation}
%$$
where the constant $C$ is independent of $\beta>0$. Together with
the result of Lemma \ref{Lem3.w}, we finally arrive at the refined estimate
%$$
\begin{equation}\label{3.ffrac5}
\|f(u(t))-f(v(t))\|_{H^{-1+\delta}}\le C\beta\|r(t)\|_{H^{\delta}}+C_\beta.
\end{equation}
%$$
Crucial for us is  that the constant $C$ is independent of $\beta$, so the coefficient
in front of $\|r(t)\|_{H^\delta}$ can be made arbitrary small by the choice of $\delta$.
\par
We are now ready to complete the proof of the proposition. To this end, we treat equation \eqref{3.contr} as
a linear \eqref{2.au} interpreting the term\\ $f(u(t))-f(v(t))$ as a part of external force $g(t)$ and
use estimate \eqref{2.bad} with  $K_\delta=-\alpha<0$, see Corollary \ref{CorA.main}. This gives
%$$
\begin{equation}
\|r(t)\|_{\bar H^\delta}\le C\|r(0)\|_{\bar H^\delta}e^{-\alpha t}+
C_\beta+C\beta\int_0^t e^{-\alpha(t-\tau)}\|r(\tau)\|_{\bar H^\delta}\,d\tau.
\end{equation}
%$$
Fixing now $\beta>0$ in such a way that $C\beta=\frac\alpha2$ and applying the Gronwall inequality,
we end up with the desired estimate
$$
\|r(t)\|_{\bar H^\delta}\le C\|r(0)\|_{\bar H^\delta}e^{-\alpha t/2}+C_1.
$$
Combining this estimate with \eqref{3.wr}, we end up with \eqref{3.notgr} and finish the
proof of the proposition.
\end{proof}
We now summarize our results concerning the truncated system \eqref{1.trunc} under the assumptions \eqref{1.fmon} and
 \eqref{1.g} for the nonlinearity $f$ and
the external force $g(t)$. We first mention that the global well-posedness and  dissipativity of
this problem in the space $\bar L^2(\Omega)$
can be obtained exactly as in Theorems \ref{Th1.E-dis} and \ref{Th1.un}, so we have the estimate
%$$
\begin{equation}\label{3.trunc-dis}
\|p(t)\|_{\bar L^2}^2+\|u(t)\|^2_{H^1}\le Q(\|p(0)\|_{\bar L^2})e^{-\alpha t}+Q(\|g\|_{L^\infty}),
\end{equation}
%$$
where positive constant $\alpha$ and monotone function $Q$ are independent of $p$ and $t$.
\par
Thus, problem \eqref{1.trunc} can be considered independently of problem \eqref{1.main} on the
whole phase space $\bar L^2(\Omega)$ and estimate \eqref{3.trunc-dis} gives us the existence of
an absorbing ball in $\bar L^2(\Omega)$, so the key assumptions \eqref{3.R} will be
automatically satisfied if we take the initial data from this absorbing ball.
\par
Let us denote by $\Cal U(t):\bar L^2(\Omega)\to\bar L^2(\Omega)$ the solution operator for
 problem \eqref{1.trunc}:
%$$
\begin{equation}
\Cal U(t)p(0):=p(t),
\end{equation}
%$$
where $p(t)$ is a solution of \eqref{1.trunc}. Then, taking into the account that the $u(t)$-component
of the
solution can be restored in a unique way (due to Lemma \ref{Lem3.w}) if the $p(t)$-component
is known, we can reformulate the results of Propositions \ref{Prop3.dis} and \ref{Prop3.contr} as follows.

\begin{corollary}\label{Cor3.exp} Let the nonlinearity $f$ satisfy \eqref{1.fmon},
\eqref{0.f} and \eqref{0.fstr} and the function $g$ satisfy \eqref{1.g}.
 Then, for a sufficiently large $R$, the $R$-ball
$\Cal B_R^\delta$ of radius $R$ in $\bar H^\delta(\Omega)$ is an exponentially attracting for the solution
 operator $\Cal U(t)$, i.e., there exists positive constant $\alpha>0$ and a monotone
  function $Q$ such that, for every bounded set $B\subset\bar L^2(\Omega)$,
%$$
\begin{equation}\label{3.expd}
\dist_{\bar L^2}(\Cal U(t)B,\Cal B_R^\delta)\le Q(\|B\|_{\bar L^2})e^{-\alpha t},
\end{equation}
%$$
where $\dist_H(A,B)$ stands for the non-symmetric Hausdorff distance between the sets $A$
 and $B$ in a Banach space $H$.
\end{corollary}
We also have the analogue of the dissipative estimate \eqref{3.trunc-dis} in the space
 $H^\delta$ for any exponent   $\delta\in[0,\frac12)$.

\begin{corollary}\label{Cor3.hd} Let the assumptions of Corollary \ref{Cor3.exp} hold and let
 $p(0)\in \bar H^\delta(\Omega)$ for some $\delta\in[0,\frac12)$. Then the following
 dissipative estimate holds for the solution of problem \eqref{1.trunc}:
 %$$
 \begin{equation}\label{3.disd}
\|p(t)\|_{\bar H^\delta}^2+\|u(t)\|^2_{H^{1+\delta}}\le
Q(\|p(0)\|_{\bar H^\delta})e^{-\alpha t}+Q(\|g\|_{L^\infty(\R_+,L^2)})
 \end{equation}
 %$$
 for some positive $\alpha$ and monotone function $Q$ which are independent of $p$ and $t$.
\end{corollary}
Indeed, this estimate can be proved analogously to the proof of Proposition \ref{Prop3.dis}, but
 even simpler since we may take $q(t)=v(t)=0$, so we leave the details to the reader.
 \par
Thus, we have verified that the solution operator $\Cal U(t)$ is well-defined and dissipative in
$\bar H^\delta(\Omega)$ for any $0\le\delta<\frac12$. It also worth to note that all of the
estimates obtained so far uses only that
%$$
\begin{equation}\label{7.g}
\|g\|_{L^\infty(R_+,H^{-1+\delta})}\le C,\ \ \delta\in(0,\frac12).
\end{equation}
%$$
 The natural next step is to extend this
 result to $\delta=1$ using bootstrapping arguments. The situation here is much simpler than for
 the first step since the nonlinearity $f$ is subcritical in the phase space $\bar H^\delta(\Omega)$,
 so the linear splitting may be used. Moreover, due to the embedding theorem
  $H^{1+\frac15}\subset L^{10}$ and the growth restrictions on $f$, we have
  %$$
  \begin{equation}\label{3.fcont}
\|f(u)\|_{L^2}\le C(1+\|u\|_{H^{1+\delta}}^5),\ \ \delta\ge\frac15
  \end{equation}
  %$$
  and, therefore, only one more step of iterations is necessary to reach $\delta=1$. Namely, we split the
   solution $(p,u)$ as follows:
   $$
   p(t)=p_1(t)+p_2(t),\ \ u(t)=u_1(t)+u_2(t),
   $$
where the decaying component $(p_1(t),u_1(t))$ solves
%$$
\begin{equation}\label{3.bcon}
\Dt p_1+\divv(Du_1)=0,\ \ -\Dx u_1+\Nx p_1=0,\ \ p_1\big|_{t=0}=p\big|_{t=0}
\end{equation}
%$$
and the smooth component $(p_2(t),u_2(t))$ is a solution of
%$$
\begin{equation}\label{3.bcom}
\Dt p_2+\divv(Du_2)=0,\ \ -\Dx u_2+\Nx p_2=g(t)-f(u(t)),\ \ p_2\big|_{t=0}=0.
\end{equation}
%$$
Then, the following proposition holds.

\begin{proposition}\label{Prop3.boot} Let $\delta\in[\frac15,\frac12)$ and let the initial
data $p(0)$ belongs to the absorbing ball $\Cal B_R^\delta$. Then the following estimates hold
 for the solutions of \eqref{3.bcon} and \eqref{3.bcom}:
 %$$
 \begin{equation}
\|p_1(t)\|_{\bar H^{\delta}}+\|u_1(t)\|_{H^{1+\delta}}\le C\| p(0)\|_{\bar H^\delta}e^{-\alpha t}
\end{equation}
%$$
and
%$$
\begin{equation}
\|p_2(t)\|_{\bar H^{1}}+\|u_2(t)\|_{H^{2}}\le
 C\|p(0)\|_{\bar H^\delta}e^{-\alpha t}+C_R(1+\|g\|_{L^\infty(L^2)}),
 \end{equation}
 %$$
 where $\alpha>0$ and $C$, $C_R$ are independent of $u$, $p$ and $t$.
\end{proposition}
Indeed, these estimates follow immediately from estimate \eqref{2.bad} with $K_\delta=-\alpha<0$ for the linear
 equation, dissipative estimate \eqref{3.disd} and estimate \eqref{3.fcont}.
\par
Analogously to Corollary \ref{Cor3.hd}, this result  gives the dissipativity
in the phase space $\bar H^1$.

\begin{corollary}\label{Cor3.h1} Let the assumptions of Corollary \ref{Cor3.exp} hold and let
 $p(0)\in \bar H^1(\Omega)$. Then the following
 dissipative estimate holds for the solution of problem \eqref{1.trunc}:
 %$$
 \begin{equation}\label{3.dish1}
\|p(t)\|_{\bar H^1}^2+\|u(t)\|^2_{H^{2}}\le
Q(\|p(0)\|_{\bar H^1})e^{-\alpha t}+Q(\|g\|_{L^\infty(\R_+,L^2)})
 \end{equation}
 %$$
 for some positive $\alpha$ and monotone function $Q$ which are independent of $p$ and $t$.
\end{corollary}

Indeed,  to get this estimate, it is enough to estimate the $L^2$-norm of $f(u)$ using
Corollary \ref{Cor3.hd} and get the desired estimate for the $H^1$-norm from the linear
equation \eqref{2.au} treating $f(u(t))$ as a part of the external forces.
\par
Analogously to Corollary \ref{Cor3.exp} the result of Proposition \ref{Prop3.boot} can be rewritten
 in the following form.

\begin{corollary}\label{Cor3.exp1} Let the  assumptions of Corollary \ref{Cor3.exp} hold. Then, for a
 sufficiently large $R$, the $R$-ball
$\Cal B_R^1$ of radius $R$ in $\bar H^1(\Omega)$ is an exponentially
attracting for the solution
 operator $\Cal U(t)$ in $\bar H^\delta$, i.e., there exists positive constant
  $\alpha>0$ and a monotone
  function $Q$ such that, for every bounded set $B\subset\bar H^\delta(\Omega)$,
%$$
\begin{equation}\label{3.exp1}
\dist_{\bar H^\delta}(\Cal U(t)B,\Cal B_R^1)\le Q(\|B\|_{\bar H^\delta})e^{-\alpha t}.
\end{equation}
%$$
\end{corollary}
Moreover, using the Lipschitz continuity of $\Cal U(t)$ in $\bar L^2(\Omega)$,
exponential attractions \eqref{3.expd} and \eqref{3.exp1} together with the transitivity
 of exponential attraction (see \cite{FGMZ}), we arrive at the following result.

 \begin{corollary}\label{Cor3.exph1} Let the  assumptions of Corollary \ref{Cor3.exp} hold.
  Then, for a sufficiently
  large $R$, the $R$-ball
$\Cal B_R^1$ of radius $R$ in $\bar H^1(\Omega)$ is an exponentially attracting for the solution
 operator $\Cal U(t)$ in $\bar L^2(\Omega)$, i.e., there exists positive constant $\alpha>0$ and a monotone
  function $Q$ such that, for every bounded set $B\subset\bar L^2(\Omega)$,
%$$
\begin{equation}\label{3.exph1}
\dist_{\bar L^2}(\Cal U(t)B,\Cal B_R^1)\le Q(\|B\|_{\bar L^2})e^{-\alpha t}.
\end{equation}
%$$
\end{corollary}
We conclude this section by translating the obtained results for the truncated system \eqref{1.trunc} to the
 initial problem \eqref{1.main}. The next result can be considered as the main result of this section.

\begin{theorem}\label{Th3.main} Let the  assumptions of Theorem \ref{Th1.E-dis} hold. Then the $R$-ball $\Bbb B_R^1$ in
 the higher energy space
 $$
 E^1:=[H^2(\Omega)\cap H^1_0(\Omega)]\times \bar H^1(\Omega)
 $$
 is an exponentially attracting set for the solution semigroup $S(t): E\to E$ generated by the problem
  \eqref{1.main} if $R$ is large enough, i.e., there exists $\alpha>0$ and monotone $Q$
  such that, for every bounded set $B\subset E$,
  %$$
  \begin{equation}\label{3.mainattr}
\dist_E(S(t)B,\Bbb B_R^1)\le Q(\|B\|_E)e^{-\alpha t}.
  \end{equation}
  %$$
Moreover, the problem \eqref{1.main} is well-posed and dissipative in the space $E^1$ as well,
i.e., if $(u(0),p(0))\in E^1$ then the following estimate holds:
%$$
\begin{equation}\label{3.maindis}
\|(u(t),p(t))\|_{E^1}\le Q(\|u(0),p(0))\|_{E^1})e^{-\alpha t}+Q(\|g\|_{L^2})
\end{equation}
%$$
for some positive $\alpha$ and monotone $Q$.
\end{theorem}
\begin{proof} Indeed, the exponential attraction \eqref{3.mainattr} follows immediately
from Corollary \ref{Cor3.exph1} and smoothing property of Corollary \ref{Cor1.smdt}.
\par
To get the dissipative estimate \eqref{3.maindis}, we note that if the initial data $(u(0),p(0))\in E^1$,
we have from equations \eqref{1.main} that
$$
\|u(0)\|_{C}+\|\Dt u(0)\|_{L^2}+\|\Dt p(0)\|_{\bar L^2}\le Q(\|(u(0),p(0))\|_{E^1}),
$$
so, we need not to use multiplication by $t$ and $t^2$ in the estimates given in the proof
 of Theorem \ref{Th1.sm-ins} in order to remove the initial data and this gives us better analogue
  of estimate \eqref{1.ins-sm}:
%$$
\begin{equation}\label{3.dtdis}
\|\Nx u(t)\|_{L^2}+\|\Dt u(t)\|_{L^2}\le Q(\|(u(0),p(0))\|_{E^1})e^{-\alpha t}+Q(\|g\|_{L^2}).
\end{equation}
%$$
This, in turn, allows to use the truncated system \eqref{1.trunc} starting from $t=0$. Then the
 desired dissipative estimate follows from the analogous estimate \eqref{3.dish1} for the truncated
  system. Thus, the theorem is proved.
\end{proof}

\section{Attractors}\label{s3}

In this section, we use the results obtained above for constructing global and exponential
attractors for problem \eqref{1.main}. We start with a global attractor.

\begin{definition}\label{Def4.attr} Let $S(t):E\to E$, $t\ge0$ be a semigroup. Then, a set
$\Cal A\subset E$ is a global attractor for $S(t)$ in $E$ if
\par
1. $\Cal A$ is compact in $E$;
\par
2. $\Cal A$ is strictly invariant, i.e., $S(t)\Cal A=\Cal A$ for all $t\ge0$.
\par
3. $\Cal A$ is an attracting set for $S(t)$ in $E$. The latter means that for every bounded set
 $B$ in $E$ and every neighbourhood $\Cal O(\Cal A)$ of the set $\Cal A$ there exist
 $T=T(B,\Cal O)$ such that
%$$
\begin{equation}
S(t)B\subset \Cal A,\ \ \forall t\ge T.
\end{equation}
%$$
If $S(t)$ is a solution semigroup related with an evolutionary equation, then the attractor $\Cal A$ of
 $S(t)$ is often called and attractor of this evolutionary equation, see \cite{BV,cv,Lad,MZ,Tem} for more details.
\end{definition}

\begin{theorem}\label{Th4.attr} Let the  assumptions of Theorem \ref{Th1.E-dis} hold. Then equation \eqref{1.main} possesses an
attractor $\Cal A$ in $E$ which is a bounded set of $E^1$. Moreover, this attractor possesses the
 following description:
%$$
\begin{equation}\label{4.rep}
\Cal A=\Cal K\big|_{t=0},
\end{equation}
%$$
where $\Cal K\subset L^\infty(\R,E)$ is a set of all complete (=defined for all $t\in\R$) bounded in $E$
solutions of equation \eqref{1.main}.
\end{theorem}
\begin{proof} According to the abstract attractor's existence theorem, see e.g., \cite{BV}, we need to
verify two properties:
\par
1. The operators $S(t)$ are continuous for every frxed $t$ as operators from $E$ to $E$;
\par
2. The semigroup $S(t)$ possesses a compact attracting set in $E$.
\par
The first property is verified in Theorem \ref{Th1.un} and the second
 one follows from Theorem \ref{Th3.main}. Since the attractor is always a subset of
  a compact attracting set, we get the boundedness of $\Cal A$ in $E^1$ and the
   representation formula \eqref{4.rep} also follows from the abstract attractor's
    existence theorem. Thus, the theorem is proved.
\end{proof}
 We now turn to exponential attractors. These objects have been introduced in \cite{EFNT} in order
to overcome the major drawback of the theory of global attractors, namely, the fact that the
rate of attraction to a global attractor may be arbitrarily slow and that there is no way in general
 to control this rate of attraction in terms of physical parameters of the considered equation. This makes
  the global attractor sensitive to perturbations and it becomes in a sense unobservable
   in finite-time simulations, see \cite{EFNT,EMZ00,EMZ05,MZ} for more details. We start with the formal definition.

   \begin{definition}\label{Def4.exp} A set $\Cal M\subset E$ is an exponential attractor for
   the semigroup $S(t): E\to E$, $t\ge0$, if
   \par
   1. $\Cal M$ is a compact set in $E$;
   \par
   2. $\Cal M$ is semi-invariant $S(t)\Cal M\subset\Cal M$ for $t\ge0$;
   \par
   3. $\Cal M$ has a finite box-counting dimension in $E$:
   $$
   \dim_F(\Cal A,E)\le C<\infty;
   $$
   \par
   4. There exist positive constant $\alpha$ and monotone function $Q$ such that, for every
   bounded set $B\subset E$, we have
   %$$
\begin{equation}\label{4.expa}
\dist_E(S(t)B,\Cal M)\le Q(\|B\|_E)e^{-\alpha t}
\end{equation}
%$$
for all $t\ge0$.
   \end{definition}
The next theorem can be considered as the main result of this section.

\begin{theorem}\label{Th4.main} Let the assumptions of Theorem \ref{Th1.E-dis} hold. Then equation \eqref{1.main}
 possesses an exponential attractor $\Cal M$ in $E$ which is a bounded set in the space $E^1$.
\end{theorem}
\begin{proof} Following the general strategy, see \cite{EMZ00,EMZ05,FGMZ,MZ},
 we first construct a {\it discrete}
 exponential attractor $\Cal M_d\subset E^1$ for the semigroup $S_n=S_1^n$ generated by the map $S(T)$
  restricted to the  $R$-ball $\Cal B_R^1$ in $E^1$. Here we fix $T>0$ in such a way that
  $$
  S(T):\Bbb B_R^1\to\Bbb B_R^1.
  $$
  It is possible to do due to estimate \eqref{3.maindis}. If the discrete attractor $\Cal M_d$
   is constructed its continuous analogue $\Cal M\subset E^1$ is given by the standard formula
   %$$
\begin{equation}\label{4.expdc}
   \Cal M:=\cup_{t\in[0,T]}S(t)\Cal M_d.
\end{equation}
   %$$
   This, together with \eqref{3.maindis} gives us the attraction property in $E$ for all
    bounded sets of $E^1$. Combining this with the exponential attraction \eqref{3.mainattr} and
     transitivity of exponential attraction (see \cite{FGMZ}), we get the desired exponential attraction of
      any bounded set in $E$. The semi-invariance follows immediately from semi-invariance of
       a discrete attractor and the explicit formula \eqref{4.expdc}. The compactness and
       finite-dimensionality also follow from \eqref{4.expdc} if we know, in addition,
       that $(t,\xi)\to S(t)\xi$ is Lipschitz (or H\"older) continuous as a map
        from $[0,T]\times\Cal M_d\to E$. The Lipschitz continuity with respect
         to the initial data is verified in Theorem \ref{Th1.un} and the Lipschitz continuity
          in times follows from the fact that $\|\Dt u(t)\|_{L^2}$ and $\|\Dt p(t)\|_{\bar L^2}$
          are uniformly bounded on $\Bbb B_R^1$ (due to estimate \eqref{3.dtdis}.
           Thus, we only need to verify the existence of a discrete exponential attractor
           $\Cal M_d$ on a set $\Bbb B_R^1$. To this end, we need the following standard
           result on the existence of exponential attractors, see \cite{EMZ00,EMZ05,MZ}.

\begin{lemma}\label{Lem4.expattr} Let $E$ and $V$ be two B-spaces such that $V$ is compactly embedded in $E$ and
 let $\Bbb B\subset E$ be a bounded set in $E$. Assume also that we are given a map
  $S:\Bbb B\to\Bbb B$ such that, for every two points $\xi_1,\xi_2\in\Bbb B$, we have a splitting
  %$$
  \begin{equation}\label{4.sp}
S(\xi_1)-S(\xi_2)=\hat\xi+\tilde\xi,
  \end{equation}
  %$$
where
%$$
\begin{equation}\label{4.con}
\|\hat\xi\|_{E}\le\kappa\|\xi_1-\xi_2\|_E
\end{equation}
%$$
for some $\kappa<\frac12$ and
%$$
\begin{equation}\label{4.smo}
\|\tilde\xi\|_{V}\le K\|\xi_1-\xi_2\|_E,
\end{equation}
%$$
where $\kappa$ and $K$ are independent of $\xi_1$ and $\xi_2$. Then the discrete semigroup
 generated by iterations of the map $S$ possesses an exponential
  attractor $\Cal M_s\subset B$ on $B\subset E$.
\end{lemma}
To apply this lemma, we need to split the solution $(\bar u(t),\bar p(t))$ of system
 \eqref{1.dif} for differences of two solutions of system \eqref{1.main} on a sum of
 contracting $(\hat u(t),\hat u(t))$
 and  smoothing $(\tilde u(t),\tilde p(t))$ components. The first part will solve
 the homogeneous linear system:
%$$
\begin{equation}\label{4.hat}
\begin{cases}
\Dt\hat u-\Dx\hat u+\Nx\hat p=0,\ \ \hat u\big|_{t=0}=\bar u\big|_{t=0},\\
\Dt\hat p+\divv(D\hat u)=0,\ \ \hat p\big|_{t=0}=\bar p\big|_{t=0}
\end{cases}
\end{equation}
%$$
and the smoothing component is taken as a solution of
%$$
\begin{equation}\label{4.tilde}
\Dt\tilde u-\Dx\tilde u+\Nx\tilde p=-l(t)\bar u,\ \ \Dt\tilde p+\divv(D\tilde u)=0,\ \
\tilde u\big|_{t=0}=\tilde p\big|_{t=0}=0,
\end{equation}
%$$
where $l(t):=\int_0^1f'(\tau u_1+(1-\tau u_2))\,d\tau$. We recall that, according to Theorem~\ref{Th1.un},
%$$
\begin{equation}\label{4.lip}
\|\bar u(t)\|_{L^2}+\|\bar p(t)\|_{\bar L^2}\le Ce^{Kt}\(\|\bar u(0)\|_{L^2}+\|\bar p(0)\|_{\bar L^2}\).
\end{equation}
%$$
Moreover, since $(u_i(0),p_i(0))\in\Bbb B_R^1$, $i=1,2$,  due to \eqref{3.dtdis} the $C$-norm of $u_i(t)$ is
uniformly bounded and, therefore,
%$$
\begin{equation}\label{4.per}
\|l(t)\bar u(t)\|_{L^2}\le Ce^{Kt}\(\|\bar u(0)\|_{L^2}+\|\bar p(0)\|_{\bar L^2}\),
\end{equation}
%$$
so the term $l(t)\bar u(t)$ can be treated as an  external force. Estimates \eqref{4.con} and
 \eqref{4.smo} are verified in the next two lemmas.

 \begin{lemma} Let the above assumptions hold. Then, the solution $(\hat u(t),\hat p(t))$
  of problem \eqref{4.hat} satisfies the estimate:
  %$$
  \begin{equation}\label{4.decc}
  \|\hat u(t)\|^2_{L^2}+\|\hat p(t)\|^2_{\bar L^2}\le
  Ce^{-\alpha t}\(\|\bar u(0)\|_{L^2}^2+\|\bar p(0)\|^2_{\bar L^2}\),
  \end{equation}
where the positive constants $C$ and $\alpha$ are independent of $u_i$ and $p_i$.
 \end{lemma}
 \begin{proof}[Proof of the lemma] Indeed, multiplying the first equation of \eqref{4.hat}
 by $D\hat u$ integrating with respect to  $x$ and using the second equation, we arrive at
 $$
 \frac12\frac d{dt}\(\|\hat u(t)\|^2_{L^2_D}+\|\hat p(t)\|^2_{\bar L^2}\)+\|\Nx u(t)\|^2_{L^2_D}=0.
 $$
 Moreover, multiplying the first equation on $-\mathfrak B\hat p(t)$ and
 using again the second equation we get
 $$
-\frac d{dt}(\hat u(t),\mathfrak B\hat p(t))+\|\hat p\|^2_{\bar L^2}-(\hat u(t),\mathfrak B\divv(D\hat u(t))=0.
$$
Multiplying this equation by small positive $\eb$ and taking a sum with the previous equation,
we finally get
$$
\frac12\frac d{dt}\(\|\hat u(t)\|^2_{L^2_D}-2\eb(\hat u(t),\mathfrak B\hat p(t))+\|\hat p(t)\|^2_{\bar L^2}\)+
\alpha\|\hat u(t)\|^2_{L^2_D}+\eb\|\hat p(t)\|^2_{\bar L^2}\le0
$$
for some positive $\alpha$. The Gronwall inequality applied to this relation gives the desired
 result if $\eb>0$ is small enough. Thus, the lemma is proved.
 \end{proof}

\begin{lemma}Let the above assumptions hold. Then, the solution $(\tilde u(t),\tilde p(t))$
 satisfies the following estimate:
%$$
\begin{equation}\label{4.sms}
\|\tilde u(t)\|_{H^1}^2+\|\tilde p(t)\|_{\bar H^1}^2\le Ce^{Kt}
\(\|\bar u(0)\|^2_{L^2}+\|\bar p(0)\|^2_{\bar L^2}\),
\end{equation}
%$$
where the constants $C$ and $K$ depend on $R$, but are independent of $u_i$ and~$p_i$.
\end{lemma}
\begin{proof}[Proof of the lemma] Indeed, multiplying the second equation of \eqref{4.tilde}
 by $\Dx \tilde u$ and using \eqref{4.per}, we get
 $$
 \frac12\frac d{dt}\|\tilde u(t)\|^2_{H^1}+\|\Dx\tilde u(t)\|^2_{H^2}
 \le C\|\tilde p(t)\|^2_{H^1}+
 Ce^{Kt}\(\|\bar u(0)\|^2_{L^2}+\|\bar p(0)\|^2_{L^2}\).
$$
Taking now $\Nx$ from the both sides of the second equation of \eqref{4.tilde} and multiplying it
 by $\Nx \tilde p(t)$, we arrive at
 $$
 \frac12\frac d{dt}\|\Nx \tilde p(t)\|^2_{L^2}\le \|\Dx\tilde u(t)\|^2_{L^2}+C\|\Nx \tilde p(t)\|^2_{L^2}.
 $$
Taking a sum of the obtained inequalities, we finally infer that
%$$
\begin{multline}
\frac12\frac d{dt}\(\|\Nx \tilde u(t)\|^2_{L^2}+\|\Nx \tilde p(t)\|^2_{L^2}\)\\\le
 C\|\Nx\tilde p(t)\|^2_{L^2}+Ce^{Kt}\(\|\bar u(0)\|^2_{L^2}+\|\bar p(0)\|^2_{\bar L^2}\)
\end{multline}
%$$
and the Gronwall inequality applied to this relation finishes the proof of the lemma.
\end{proof}
We are now ready to complete the proof of the theorem. Indeed, estimates \eqref{4.decc} and
\eqref{4.sms} guarantee that the assumptions of Lemma \ref{Lem4.expattr} are satisfied if we take
$$
V:=H^1_0(\Omega)\times \bar H^1(\Omega)
$$
and fix $T$ big enough that $Ce^{-\alpha T}<\frac12$. Thus, the discrete exponential
 attractor $\Cal M_d$ is constructed and the desired continuous exponential attractor $\Cal M$
 can be constructed via \eqref{4.expdc} as explained above. Therefore, the theorem is proved.
\end{proof}

\section{Generalizations and concluding remarks}\label{s7}

In this section, we briefly discuss the so-called Navier-Stokes-Brinkman-Forchheimer
equation in the following form:
%$$
\begin{equation}\label{7.NS}
\begin{cases}
\Dt u+B(u,u)-\Dx u+\Nx p+f(u)=g,\ u\big|_{\partial\Omega}=0,\ u\big|_{t=0}=u_0,\\
\Dt p+\divv(u)=0,\ p\big|_{t=0}=p_0,
\end{cases}
\end{equation}
%$$
where
%$$
\begin{equation}\label{7.B}
B(u,v)=(u,\Nx)v+\frac12\divv(u)v.
\end{equation}
%$$
The extra term $\frac12 \divv(u)u$ is added to the standard Navier-Stokes inertial term in
order to preserve the energy identity, see \cite{FG,GT,Tem1} and references therein. Indeed,
in this case we have
$$
(B(u,v),v)\equiv 0,\ \ \forall u,v\in H^1_0(\Omega)
$$
and we have the energy identity \eqref{1.eq} with $D=1$ exactly as in the case $B=0$
considered above, namely,
%$$
\begin{equation}\label{7.eq}
\frac12\frac d{dt}\(\|u\|^2_{L^2}+\|p\|^2_{\bar L^2}\)+\|\Nx u\|^2_{L^2}+(f(u),u)=(g,u).
\end{equation}
%$$
The theory of this equation is very similar to the case $B=0$ considered above with the only
difference that, in order to control the extra non-linearity $B$, we need to assume that $f(u)$
has a super-cubic growth rate, see  \cite{HR,KZ}, but this assumption
 is already incorporated to \eqref{0.f} if $l>1$.
\par
We start with the analogue of dissipative estimate \eqref{1.E-dis}. The analogue of \eqref{1.eq} is
 already obtained, so in order to get the key differential inequality \eqref{1.Gp}, we
  only need to estimate the extra term
%$$
\begin{multline}
|\eb(B(u,u),\mathfrak Bp)|\le C\eb\|u\|_{L^3}\|\Nx u\|_{L^2}\|p\|_{\bar L^2}\le
C\|u\|_{L^3}^3+\\+\frac12\|\Nx u\|_{L^2}^2+C\eb\|p\|^6_{\bar L^2}\le
 \frac12(f(u),u)+C+\frac12\|\Nx u\|^2_{L^2}+C\eb^6\Cal E_\eb(u,p)^3,
\end{multline}
%$$
where the constant $C$ is independent of $\eb$. Thus, analogously to Theorem \ref{Th1.E-dis}, we have
 the following result.

 \begin{proposition}\label{Prop7.E-dis} Let the assumptions of Theorem \ref{Th1.E-dis} hold, $l>1$ and let
  $(u,p)$ be a weak energy solution of problem \eqref{7.NS}. Then, this solution satisfies the
  dissipative estimate \eqref{1.E-dis}.
 \end{proposition}
Let us now turn to uniqueness. This can be proved exactly as in the incompressible case (see \cite{KZ}).
Indeed, in comparison with Theorem \ref{Th1.un}, we need to estimate the extra term
$$
(B(u_1,u_1)-B(u_2,u_2),u_1-u_2)=(B(\bar u,u_2),\bar u),\ \ \bar u=u_1-u_2,
$$
where $u_1$ and $u_2$ are two solutions of \eqref{7.NS}. Integrating by parts and
 using the Cauchy-Schwarz inequality, we get
 $$
|B(\bar u,u_2),\bar u)|\le \frac14\|\Nx \bar u\|^2_{L^2}+C\|u_2\bar u\|^2_{L^2}.
 $$
On the other hand, using assumptions \eqref{0.f}, analogously \cite{KZ}, we get
$$
(f(u_1)-f(u_2),\bar u)\ge \kappa(|u_1|^{1+l}+|u_2|^{1+l},|\bar u|^2)-L\|\bar u\|^2_{L^2}\ge
C\|u_2\bar u\|^2-\tilde L\|\bar u\|^2_{L^2}
$$
and, therefore,
$$
(B(u_1)-B(u_2),\bar u)+(f(u_1)-f(u_2),\bar u)\ge -\tilde L\|\bar u\|^2_{L^2}.
$$
Arguing further as in the proof of Theorem \ref{Th1.un}, we get estimate \eqref{1.lip} and
 verify the uniqueness of the solution for problem \eqref{7.NS}. Thus, as in the case of $B=0$,
  equation \eqref{7.NS}, generates a dissipative semigroup $S(t)$ in the phase space $E$.
\par
We now discuss the smoothing property and start with the instantaneous smoothing
 (analog of Theorem \ref{Th1.sm-ins}).
\begin{proposition}\label{Prop7.sm-ins} Let the  assumptions of Theorem \ref{Th1.E-dis} hold, $l>1$ and
let $(u,p)$ be a weak
 energy solution of equations \eqref{7.NS}.
Then the following partial smoothing property holds:
%$$
\begin{multline}\label{7.sm}
t^{8/3}\|\Nx u(t)\|^2_{L^2}+t^{8/3}\|\Dt u(t)\|^2_{L^2}+t^{8/3}\|\Dt p(t)\|^2_{\bar L^2}+\\+
\int_0^{t}s^{8/3}\|\Nx\Dt u(s)\|^2_{L^2}\,ds\le Q(\|(u(0),p(0)\|^2_{E})+Q(\|g\|_{L^2}),
\end{multline}
%$$
where $t\in[0,1]$ and a function $Q$ is independent of $t$ and $u$.
\end{proposition}
\begin{proof}
 Here, we have a little difference (in comparison with the proof of Theorem \ref{Th1.sm-ins}),
 namely, multiplication of the equation
 on $\Dt u$ does not work since we have not enough regularity to control the term $(B(u,u),\Dt u)$.
 By this reason, again similarly to the incompressible case (see \cite{KZ}), we need to differentiate
 the first equation of \eqref{7.NS} with respect to  $t$ and multiply it by $v=\Dt u$ at the first step.
 The nonlinearity
  $B(u,v)+B(v,u)$ is controlled here by the second nonlinearity $f'(u)v$ exactly as in the
   proof of uniqueness, so we get the following analogue of
%$$
\begin{equation}\label{7.dtv}
\frac12\frac d{dt}\(\|v(t)\|^2_{L^2}+\|\Dt p\|^2_{\bar L^2}\)+\|\Nx v\|^2_{L^2}\le \tilde L\|v\|^2_{L^2}.
\end{equation}
%$$
Moreover, from the first equation of \eqref{7.NS} and the dissipative estimate \eqref{1.E-dis},
 we infer after the standard estimates that
%$$
\begin{equation}\label{7.dtu}
\|v\|_{L^{6/5}(0,1;H^{-1})}\le Q(\|(u_0,p_0)\|_{E})+Q(\|g\|_{L^2}).
\end{equation}
%$$
Estimate \eqref{7.dtu} replaces the missed control of the quantity $\int_0^ts\|v(s)\|_{L^2}^2\,ds$ and
allows us to get the desired smoothing property. Indeed, multiplying  \eqref{7.dtv} by $t^{8/3}$, integrating
 in time and using the estimate
 %$$
\begin{multline}
 \int_0^t s^{5/3}\|v(s)\|^2_{L^2}\,ds \le \\\le
  \int_0^s(s^{4/3}\|v(s)\|_{L^2})^{1/2}(s^{4/3}\|v(s)\|_{H^1})^{3/4}\|v(s)\|^{3/4}_{H^{-1}}\,ds \le \\\le
  \eb\sup_{s\in[0,t]}\left\{s^{8/3}\|v(s)\|_{L^2}^2\right\}+\eb\int_0^ts^{8/3}\|v(s)\|^2_{H^1}\,ds+
  C_\eb\|v\|_{L^{6/5}(0,1;H^{-1})}^2,
\end{multline}
%$$
where $\eb>0$ can be taken arbitrarily small, we end up with the desired smoothing property
 for the derivatives
 $$
 t^{4/3}\|\Dt u(t)\|_{L^2}+t^{4/3}\|\Dt p(t)\|_{\bar L^2}\le Q(\|u_0,p_0\|_{E})+Q(\|g\|_{L^2})
 $$
for $t\in[0,1]$ and some monotone function $Q$. Returning back to the first equation of \eqref{7.NS},
multiplying it by $u(t)$ and integrating in $x$, we get
$$
\|\Nx u(t)\|_{L^2}^2+(|f(u(t)).u(t)|,1)\le C(\|p(t)\|^2_{\bar L^2}+\|g\|^2_{L^2}+\|\Dt u(t(\|^2_{L^2})
$$
which together with the previous estimate and dissipative estimate \eqref{1.E-dis} give the
 desired smoothing property and finishes the proof of the proposition.
\end{proof}
As in the case $B=0$, this instantaneous smoothing property allows us to reduce the study of the
 asymptotic smoothing to the truncated problem
  %$$
  \begin{equation}\label{7.trunc}
  \begin{cases}
  -\Dx u+\Nx p+B(u,u)+f(u)=g(t),\ \ \ u\big|_{\partial\Omega}=0\\
  \Dt p+\divv(Du)=0,\ \ p\big|_{t=0}=p_0,\  ,
\end{cases}
  \end{equation}
  %$$
where $g(t):=g-\Dt u(t)$ satisfies \eqref{1.g}. Moreover, using the obvious estimate
%$$
\begin{equation}\label{7.Bhd}
\|B(u(t),u(t))\|_{H^{-1/2}}\le \|B(u(t),u(t))\|_{L^{3/2}}\le C\|\Nx u(t)\|^2_{L^2},
\end{equation}
%$$
we can assume without loss of generality (due to the dissipative estimate \eqref{1.E-dis} and
smoothing property \eqref{7.sm}) that the nonlinearity $B(u,u)$
is boun\-ded in $L^\infty(\R_+,H^{-1/2})$. Thus,
we may treat the nonlinearity $B(u,u)$ as a part of $g$ as well. Then the new function $g$ will
satisfy \eqref{7.g} and we may treat equations \eqref{7.trunc} exactly as equations \eqref{1.trunc}.
\par
This gives us the analogues of Corollaries \ref{Cor3.exp} and  \ref{Cor3.hd} for the truncated
system \eqref{7.trunc}. In order to make the second step of bootstrapping, we note that
$$
\| B(u(t),u(t))\|_{L^2}\le C\|u(t)\|_{H^{1+\delta}}
$$
for $\delta\ge\frac14$. Therefore, if the $H^{1+\delta}$-regularity of $u(t)$ is verified
for $\delta\ge\frac14$, the next step of bootstrapping will give us the $H^2$-regularity exactly
 as in the Section \ref{s2}. Thus, we have proved the following analogue of Theorem \ref{Th3.main}.

\begin{theorem}\label{Th7.main} Let the  assumptions of Theorem \ref{Th1.E-dis} hold and $l>1$.
Then the $R$-ball $\Bbb B_R^1$ in
 the higher energy space $E^1$
 is an exponentially attracting set for the solution semigroup $S(t): E\to E$ generated by the problem
  \eqref{7.NS} if $R$ is large enough, i.e., there exists $\alpha>0$ and monotone $Q$
  such that, for every bounded set $B\subset E$,
  %$$
  \begin{equation}\label{7.mainattr}
\dist_E(S(t)B,\Bbb B_R^1)\le Q(\|B\|_E)e^{-\alpha t}.
  \end{equation}
  %$$
Moreover, the problem \eqref{7.NS} is well-posed and dissipative in the space $E^1$ as well,
i.e., if $(u(0),p(0))\in E^1$ then the following estimate holds:
%$$
\begin{equation}\label{7.maindis}
\|(u(t),p(t))\|_{E^1}\le Q(\|u(0),p(0))\|_{E^1})e^{-\alpha t}+Q(\|g\|_{L^2})
\end{equation}
%$$
for some positive $\alpha$ and monotone $Q$.
\end{theorem}
Finally, we have the analogue of Theorem \ref{Th4.main} on exponential attractors.

\begin{theorem}\label{Th8.main} Let the assumptions of Theorem \ref{Th1.E-dis} hold and $l>1$.
Then equation \eqref{7.NS}
 possesses an exponential attractor $\Cal M$ in $E$ which is a bounded set in the space $E^1$.
\end{theorem}
The proof of this result repeats word by word the proof of Theorem \ref{Th4.main} (we have more
 than enough regularity of solutions $u_1(t)$ and $u_2(t)$ to handle the extra nonlinear term)
  and by this reason is omitted.
\par
We conclude the exposition by several remarks.

\begin{remark} \label{Rem7.2D} We have considered equations \eqref{1.main} and \eqref{7.NS}
in the most complicated 3D case only. The 2D case can be treated analogously, but it is actually
essentially simpler. Indeed, due to the Sobolev embedding $H^1\subset L^q$ for all $q<\infty$, the
control of the $H^1$-norm of the solution $u$ gives the control of the $L^2$-norm of $f(u)$ for
 any growth exponent $l$, so the restriction $l\le2$ can be removed here and any polynomial
 nonlinearity is {\it subcritical} in 2D case.
 \par
 Another simplification comes from the fact that in 2D case the inertial term $B(u,u)$ can
 be handled without the help of the nonlinearity $f(u)$, so we do not need  to require the super-cubic
 growth rate of $f(u)$. In particular, the purely Navier-Stokes case $f(u)$ is also covered
  by our theory and gives some new results here as well. For instance, in comparison with \cite{GT}, we
  get the $E^1$-regularity of the attractor for the case of Dirichlet boundary conditions as well.
\end{remark}
\begin{remark}\label{Rem7.super} An interesting question is related with the supercritical case where
 the nonlinearity grows faster than $u|u|^4$. In the case of incompressible Brinkman-Forchheimer
  equations as well as in the case of strongly damped wave equations, the restriction $l\le2$
  in \eqref{0.f} is not necessary as shown in \cite{KZ,KZ09}. Some methods developed
  there can be extended to the case of equations \eqref{1.main} as well.
\par
 For, the existence of weak solutions in this case can
  be verified based on the energy identity \eqref{1.eq}, their uniqueness follows exactly as in the proof
   of Theorem \ref{Th1.un} where only the monotonicity assumption $f'(u)\ge-L$  is actually  used.
   Moreover, the local smoothing property and estimates for $\Dt u$ stated in Theorem
   \ref{Th1.sm-ins} also work for the super-critical case as well.
   \par
   However, there is a problem here which prevents us to treat the supercritical case, namely,
   the absence of a {\it dissipative} estimate for the solution $u$ in the energy norm. Indeed,
   the derivation of such an estimate in Theorem \ref{Th1.E-dis} is based on multiplication
    of the equation by $\mathfrak Bp$, where $\mathfrak B$ is a Bogowski operator, but in the
    supercritical case we cannot do this at least in a direct way since the term
     $(f(u),\mathcal Bp)$ is out of control. We believe that this problem has a technical nature
     which can be overcome and are planning to return to the supercritical case somewhere else.
\end{remark}

\appendix

\section{An auxiliary linear problem}\label{s4}

In this appendix, we study the following linear problem:
%$$
\begin{equation}\label{2.au}
\Dt p+\divv(Du)=0,\ \ -\Dx u+\Nx p=g(t),\ \ p\big|_{t=0}=p_0,\ \ u\big|_{\partial\Omega}=0.
\end{equation}
%$$
Note that, solving the second equation of \eqref{2.au} with respect to $u$, we get
%$$
\begin{equation}
u(t)=-(-\Dx)^{-1}\Nx p+(-\Dx)^{-1}g(t),
\end{equation}
%$$
where the Laplacian is endowed with the homogeneous Dirichlet boundary condition. Inserting this
 expression to the first equation, we arrive at
%$$
\begin{equation}\label{2.smp}
\Dt p+\mathfrak Ap=\divv(D(-\Dx)^{-1}g(t)),
\end{equation}
%$$
where
%$$
\begin{equation}\label{2.a}
\mathfrak Ap:=-\divv(D(-\Dx)^{-1}\Nx p).
\end{equation}
%$$
Thus, the key question here are the properties of the operator $\mathfrak A$.

\begin{proposition}\label{Prop2.bounds} The operator
$\mathfrak A\in\Cal L(\bar H^\delta(\Omega),\bar H^\delta(\Omega))$
 if $\delta>-\frac12$. Moreover, this operator is positive definite and self-adjoint
  in $\bar L^2(\Omega)$:
 %$$
 \begin{equation}
(\mathfrak Ap,p)\ge \alpha\|p\|^2_{\bar L^2}, \ \ p\in \bar L^2(\Omega)
 \end{equation}
 %$$
for some $\alpha>0$.
\end{proposition}
\begin{proof} Indeed, the first statement is an immediate corollary of the classical elliptic
regularity estimates for the Laplacian, see e.g., \cite{Tri}, so we only need to check
 the stated properties for $\delta=0$. The fact that $\mathfrak A$ is self-adjoint is also straightforward, so
  we need to verify positiveness. Namely,
  %$$
\begin{multline}
  (\mathfrak A p,p)=-(\divv(Du),p)=(Du,\divv p)=\\=-(Du,\Dx u)=
  (D\Nx u,\Nx u)\ge \alpha_1\|\Nx u\|^2_{L^2},
\end{multline}
  %$$
  where $-\Dx u+\Nx p=0$ and $\alpha_1>0$ is the smallest eigenvalue of the matrix $D$. Using, e.g.,
   the Bogovski operator it is easy to show that $\|p\|^2_{\bar L^2}\le C\|\Nx u\|^2_{L^2}$ for
   some positive constant $C$. Thus, the proposition is proved.
\end{proof}
As an immediate corollary of this proposition, we get the following result.

\begin{corollary}\label{Cor2.bad} Let $p_0\in\bar H^\delta(\Omega)$ and
$g\in L^1(0,T;H^{\delta-1}(\Omega))$,
 $\delta>-\frac12$. Then,
the solution $p(t)$ of equation \eqref{2.smp} belongs to $\bar H^\delta$ for all $t\ge0$ and
the following estimate holds:
%$$
\begin{equation}\label{2.bad}
\|p(t)\|_{\bar H^\delta}\le C_\delta\|p(0)\|_{\bar H^\delta}e^{K_\delta t}+
C_\delta\int_0^te^{K_\delta(t-\tau)}\|g(\tau)\|_{H^{\delta-1}}\,d\tau,
\end{equation}
%$$
where the constants $C_\delta$ and $K_\delta$ depend only on $\delta$. In particular, for $\delta=0$,
the corresponding exponent $K_0=-\alpha<0$.
\end{corollary}

\begin{remark} The result of Corollary \ref{Cor2.bad} gives the dissipative estimate for $\delta=0$ only.
 For other values of $\delta$, the constant $K_\delta$ a priori may be positive, then the
 obtained estimate will
  be not dissipative. This is related with the fact that we do not know a priori that the spectrum
   of operator $\mathfrak A$ is the same in all Sobolev spaces $\bar H^\delta(\Omega)$,
   so if it depends on $\delta$,
   then it may happen that equation \eqref{2.smp} may become unstable for some values of $\delta$. We
   expect that, in a fact, the spectrum of $\mathfrak A$ is {\it independent} of $\delta$, but failed to find
    the proper reference. So, in order to avoid the technicalities, we restrict ourselves to
    the most important for our purposes case $\delta=1$ and verify that the corresponding
     $K_1$ is also negative.
\end{remark}
\begin{proposition}Let $p_0\in\bar H^1(\Omega)$ and $g\in L^1(0,T;L^2(\Omega))$. Then, the solution
 $p(t)$ of the truncated problem \eqref{2.smp} satisfied the following estimate:
 %$$
 \begin{equation}\label{2.h1dis}
\|p(t)\|_{\bar H^1}\le C\|p(0)\|_{\bar H^1}e^{-\alpha t}+C\int_0^t\|g(s)\|_{L^2}\,ds,
 \end{equation}
%$$
where the positive constants $C$ and $\alpha$ are independent of $t$ and $p$.
\end{proposition}
\begin{proof} In the case of {\it periodic} boundary conditions, the desired estimate
 can be obtained just by multiplying equation \eqref{2.smp} by $\Dx p$.
 However, this does not work in the case of Dirichlet boundary conditions because of  the presence of extra
 boundary integrals arising after integration by parts. So, in this case we will use the
 localization technique instead. Note also that we only need to verify \eqref{2.h1dis}
  for $g=0$. The general case will follow
  then form the Duhamehl formula.
\par
{\it Step 1. Interior estimates.} Let us fix a non-negative cut-off function
 $\phi(x)\in C^1_0(\R)$ such that
$\phi(x)=0$ if $x$ is in the $\mu/2$-neighbourhood of the boundary
 $\partial\Omega$ and $\phi\equiv1$
 if $x\in\Omega$ and is outside the $\mu$-neighbourhood of $\Omega$. In addition, we require that
 $$
 |\Nx \phi(x)|\le C\phi(x)^{1/2},\ \ x\in\Omega.
 $$
 It is not difficult to see that such a function exists for all $\mu>0$ small enough.
\par
We write equation \eqref{2.smp} as a system \eqref{2.au} with $g=0$ (in order to avoid the
 inverse Laplacian) and multiply the first equation by $-\divv(\phi\Nx p)$. Then, after
 integration by $x$, we get
%$$
\begin{multline}
\frac12\frac d{dt}(\phi,|\Nx p|^2)=(\divv(Du),\divv(\phi\Nx p))=\\=
\sum_{i=1}^3(\divv(Du),\partial_{x_i}(\phi\partial_{x_i}p))=
-\sum_{i=1}^3(Du,\partial_{x_i}(\phi\partial_{x_i}\Nx p))-\\-
\sum_{i=1}^3(Du,\partial_{x_i}(\Nx\phi\partial_{x_i}p)=
-\sum_{i=1}^3(\partial_{x_i}(D\partial_{x_i}u),\Nx p)+\\+
 (\divv(Du),\Nx\phi\cdot\Nx p)=
-\sum_{i=1}^3(\partial_{x_i}(D\phi\partial_{x_i}u),\Dx u)+\\+
 (\divv(Du),\Nx\phi\cdot\Nx p)=
-(\phi D\Dx u,\Dx u)-(D\Dx u\cdot\Nx\phi,\divv(u))+\\+(\divv(Du),\Nx\phi\cdot\Nx p)\le
-\alpha_1(\phi,|\Dx u|^2)+\frac{\alpha_1}4(\phi,|\Dx u|^2)+\\+\frac{\alpha_1}4(\phi,|\Nx p|^2)+
C\|\Nx u\|^2_{L^2}\le -\frac{\alpha_1}2(\phi,|\Nx p|^2)+C\|p\|^2_{\bar L^2}.
\end{multline}
%$$
Since we have already known from Corollary \ref{Cor2.bad} that
%$$
\begin{equation}\label{2.goodp}
\|p(t)\|_{\bar L^2}+\|\Nx u(t)\|^2_{L^2}\le Ce^{-\alpha t}\|p(0)\|_{L^2},
\end{equation}
%$$
then applying the Gronvall inequality to the obtained relation, we get the desired interior
 dissipative estimate:
%$$
\begin{equation}
(\phi,|\Nx p(t)|^2)\le C\((\phi,|\Nx p(0)|^2)+\|p(0)\|^2_{\bar L^2}\)e^{-\beta t},
\end{equation}
%$$
where $C$ and $\beta$ are some {\it positive} constants.
\par
{\it Step 2. Boundary estimates: tangential directions.} Let us introduce in a small neighbourhood
 of the boundary three smooth orthonormal vector fields
 $$
 \tau_3(x):=n=(n^1(x),n^2(x),n^3(x)), \ \ \tau_1(x):=(\tau_1^1(x),\tau_1^2(x),\tau_1^3(x))
 $$
 and $\tau_2(x):=(\tau_2^1(x),\tau_2^2(x),\tau_2^3(x))$ such that $n(x)$ coincides with the outer normal
  vector when $x\in\partial\Omega$ and $\tau_1(x),\tau_2(x)$ give the complement pair of
  tangential vectors. This triple of vector field may not exist globally near
   the boundary, but only locally, so being pedantic we need to use the partition of unity near
   the boundary to localize them, but we ignore this standard procedure in order to avoid technicalities
    (this localization can be done exactly in the way how we get interior estimates). After defining
     the triple of vector fields near the boundary, we use the proper scalar cut-off function in order
      to extend these fields to the whole domain~$\bar\Omega$.
      \par
      Let us define the corresponding differentiation operators along these vector fields:
      $$
      \partial_{\tau_i}u:=\sum_{j=1}^3\tau^j_i(x)\partial_{x_j}u
      $$
In contrast to the differentiation with respect to coordinate directions, these operators do not
commute in general, but their commutator is a lower order operator (again first order differential
operator):
$$
[\partial_{\tau_i},\partial_{\tau_j}]=\partial_{\{\tau_i,\tau_j\}},
$$
where $\{\tau_i,\tau_j\}$ is a Lie bracket of vector fields $\tau_i$ and $\tau_j$. This commutation
 up to lower order terms is important for our method. One more crucial fact for us is that the
 condition $u\big|_{\partial\Omega}=0$ implies that $\partial_{\tau_i}u\big|_{\partial\Omega}=0$,
  $i=1,2$ so differentiation with respect to tangential derivatives preserve the Dirichlet
  boundary conditions.
\par
We are now ready to get the desired estimates for tangential derivatives. To this end, we denote
 $q:=\partial_{\tau_i}p$ and $v=\partial_{\tau_i}u$. Then, differentiating  the equations \eqref{2.au}
 in the direction $\tau_i$, we arrive at
 %$$
 \begin{equation}\label{2.tan}
\Dt q+\divv(Dv)+M(x)\Nx u=0,\ \ -\Dx v+\Nx q=N(x)\Nx p+Ru,
 \end{equation}
 %$$
where the matrices $M$ and $N$ are smooth and $R$ is a linear second order differential
operator with smooth coefficients. Multiplying the first and second equations of \eqref{2.tan} by $q$ and $Dv$
respectively, we arrive at
%$$
\begin{multline}\label{2.tand}
\frac12 \frac d{dt} \|q\|^2_{L^2}=(Dv,\Nx q)-(M\Nx u,q)-(D\Nx v,\Nx v)-
\\-(\Nx q, Dv)+(N\Nx p,Dv)+(Ru,Dv)\le\\
\le
-\alpha_1\|\Nx v\|^2_{L^2}
+\eb\|\Nx v\|^2_{L^2}+\eb\|q\|^2_{L^2}+C_\eb(\|\Nx u\|^2_{L^2}+\|p\|^2_{L^2}),
\end{multline}
%$$
where $\eb>0$ can be arbitrarily small. Moreover, multiplying the second equation by $Bq$ after the standard estimates,
we get
%$$
\begin{equation}
\|q\|_{L^2}^2\le C\|\Nx v\|_{L^2}+C\|\Nx u\|^2_{L^2}+C\|p\|_{L^2}^2.
\end{equation}
%$$
Inserting this estimate in \eqref{2.tand} and fixing $\eb>0$ to be small enough, we finally arrive at
%$$
\begin{equation}
\frac d{dt}\|q\|_{L^2}^2+\bar\alpha\|q\|^2_{L^2}\le C\|\Nx u\|^2_{L^2}+\|p\|^2_{L^2}
\end{equation}
%$$
for some $\bar\alpha>0$. Applying the Gronwall inequality to this relation and using \eqref{2.goodp},
we have
%$$
\begin{equation}
\|\partial_{\tau_1}p(t)\|^2_{L^2}+\|\partial_{\tau_2} p(t)\|_{L^2}^2\le C e^{-\alpha t}\|\Nx p(0)\|^2_{L^2},
\end{equation}
%$$
where $\alpha>0$ and $C$ are independent of $p$ and $t$. Thus, the desired estimates for tangential
 derivatives are obtained.
 \par
 {\it Step 3. Boundary estimates: normal direction.} We now want to estimate the normal
 derivative $\partial_n p$ using equations \eqref{2.au} and the already obtained estimates for the
  tangential derivatives. To this end, we need some preparations. Let us write the vector $u$ in the form
  $$
u=u_n n+u_{\tau_1}\tau_1+u_{\tau_2}\tau_2,\ \ u_n:=u.n,\ \ u_{\tau_i}=u.\tau_i.
  $$
  Then, multiplying the second equation of \eqref{2.au} by $\tau_i$, $i=1,2$ and  using the fact that
  the $L^2$-norm of $\partial_{\tau_i}p$ as well as $H^1$-norm of $u$ are already estimated, we get
  %$$
\begin{equation}\label{2.tt}
\|u_{\tau_1}\|_{H^2}+\|u_{\tau_2}\|_{H^2}\le Ce^{-\alpha t}\|p(0)\|_{H^1}.
\end{equation}
%$$
Moreover, multiplying the second equation of \eqref{2.au} by $n$ and using that the $H^1$-norms
 of tangential derivatives of $u$ are already under the control, we arrive at
 %$$
\begin{equation}\label{2.np}
 \|\partial^2_nu_n(t)-\partial_n p(t)\|_{L^2}\le Ce^{-\alpha t}\|p(0)\|_{H^1}.
\end{equation}
 %$$
 We now return to the first equation of \eqref{2.au} (the equation for pressure). Taking the normal
 derivative from both sides of this equation and using \eqref{2.tt} and the fact that the $H^1$-norm of
 $\partial_{\tau_j}u_{n}$, $j=1,2$ are also under the control, we arrive at
 $$
 \Dt\partial_n p+(Dn.n)\partial_n^2u_n=h(t),\ \ \|h(t)\|_{L^2}\le Ce^{-\alpha t}\|p(0)\|_{H^1}.
 $$
 Multiplying the obtained equation by $\partial_n p$, integrating over $x$ and using \eqref{2.np}
 together with positivity of the matrix $D$, we finally get
 $$
 \frac d{dt}\|\partial_n(t)\|^2_{L^2}+\alpha_2\|\partial_n p(t)\|^2_{L^2}\le
  Ce^{-2\alpha t}\|p(0)\|^2_{H^1}
$$
and applying the Gronwall inequality, we get the desired estimate for the normal derivative:
$$
\|\partial_n p(t)\|^2_{L^2}\le Ce^{-\alpha t}\|p(0)\|_{H^1}^2
$$
where $\alpha>0$ and $C$ are independent of $t$ and $u$.
\par
Combining together the obtained interior, tangential and normal estimates, we derive that
$$
\|p(t)\|_{\bar H^1}^2\le Ce^{-\alpha t}\|p(0)\|^2_{\bar H^1}
$$
and finish the proof of the proposition.
\end{proof}
\begin{corollary}\label{CorA.main} Let the assumptions of Corollary \ref{Cor2.bad} hold
 and let $\delta\in[0,1]$. Then, the
 corresponding estimate \eqref{2.bad} holds with $K_\delta\le-\alpha<0$.
\end{corollary}
Indeed, we have verified this property for $\delta=0$ and $\delta=1$. For fractional values $0<\delta<1$, the result
 follows by the interpolation.

\end{document}